\documentclass{article}
\usepackage[a4paper, totalwidth= 15cm, 
        totalheight=20cm, hmarginratio=1:1]{geometry}

\usepackage{tocbasic}
\DeclareTOCStyleEntry[
  beforeskip=0.5em plus 1pt,% default is 1em plus 1pt
  pagenumberformat=\textbf
]{tocline}{section}

\let\olddiv\div
\usepackage[utf8]{inputenc}
\usepackage[T1]{fontenc}
\usepackage{mathrsfs}
\usepackage{float}
\usepackage{comment}
\usepackage{graphicx}
\usepackage{caption}
\usepackage{xspace}
\usepackage{xpunctuate}

\usepackage[toc]{appendix}
\usepackage{amssymb,amsmath,amsthm,amsfonts}
\usepackage{thmtools,mathtools}
\usepackage[dvipsnames]{xcolor}
\usepackage{tikz,tikz-cd}
\usetikzlibrary{%
  matrix,%
  calc,%
  arrows%
}
\usepackage{hyperref}
\usepackage{physics}
\usepackage{ stmaryrd }
\usepackage{leftidx}
\usepackage{cleveref}% http://ctan.org/pkg/cleveref
\usepackage{sidecap}

\newcounter{char}
\loop\ifnum\value{char}<27
    \expandafter\edef\csname\Alph{char}sc\endcsname{\noexpand\mathscr{\Alph{char}}}
    \addtocounter{char}{1}
\repeat

\newtheorem{theorem}{Theorem}[section]
\newtheorem{lemma}[theorem]{Lemma}
\newtheorem{proposition}[theorem]{Proposition}
\newtheorem{conj}[theorem]{Conjecture}
\newtheorem{corollary}[theorem]{Corollary}
\newtheorem{defn}[theorem]{Definition}
\theoremstyle{definition}
\newtheorem{remark}[theorem]{Remark}

\crefname{theorem}{Theorem}{}
\crefname{lemma}{Lemma}{}
\crefname{proposition}{Proposition}{}
\crefname{corollary}{Corollary}{}
\crefname{defn}{Definition}{}
\crefname{remark}{Remark}{}
\crefname{conj}{Conjecture}{}
\crefname{section}{\textsection}{\textsection\textsection}
\crefname{subsection}{\textsection}{\textsection\textsection}
\crefname{subsubsection}{\textsection}{\textsection\textsection}
\crefname{appendix}{Appendix}{Appendices}
\crefname{item}{condition}{conditions}

\def\bB{{\mathbf{B}}}

\def\bD{{\mathbf{D}}}

\def\bN{{\mathbf{N}}}

\def\bP{{\mathbf{P}}}

\def\bS{{\mathbf{S}}}
\def\bT{{\mathbf{T}}}

\def\cC{{\mathcal{C}}}

\def\cE{{\mathcal{E}}}
\def\cF{{\mathcal{F}}}

\def\cH{{\mathcal{H}}}

\def\cL{{\mathcal{L}}}

\def\cO{{\mathcal{O}}}

\def\cQ{{\mathcal{Q}}}

\def\cS{{\mathcal{S}}}

\def\cU{{\mathcal{U}}}

\def\wcU{\widetilde{\mathcal{U}}}

\def\wV{\widetilde{V}}

\def\Pic{\operatorname{Pic}}

\def\Cone{{\mathrm{Cone}}}
\def\Ext{\operatorname{Ext}}
\def\Gr{\operatorname{\mathsf{Gr}}}
\def\IGr{\operatorname{\mathsf{IGr}}}

\def\supp{\operatorname{\mathsf{supp}}}

\def\Stair{\operatorname{\mathsf{Stair}}}
\def\ranksf {\operatorname{\mathsf{rank}}}
\def\dimsf {\operatorname{\mathsf{dim}}}
\def\maxsf {\operatorname{\mathsf{max}}}

\def\Hom{\operatorname{Hom}}

\def\HS{\operatorname{HS}}
\def\VS{\operatorname{VS}}

\def\H{\operatorname{H}}

\def\GL{\operatorname{\mathsf{GL}}}

\newcommand*{\cf}{cf.\@\xspace}

\def\Sp{\operatorname{\mathsf{Sp}}}
\def\DX{\operatorname{\mathbf{D}^{\mathrm{b}}(\mathit{X})}}
\def\Ker{\operatorname{Ker}}

\def\Aut{\operatorname{Aut}}

\title{On the derived category of $\IGr(3,9)$}
\author{Warren Cattani}

\newcommand{\Addresses}{{
  \bigskip
  \footnotesize
  Warren~Cattani, \textsc{SISSA,
    Trieste, Italy.}\par\nopagebreak
  \textit{Email address:  }\href{mailto:wcattani@sissa.it}{\texttt{wcattani@sissa.it}}
}}

\date{}
\begin{document}
\hypersetup{
  pdfauthor={W. Cattani},
  pdftitle={Lefschetz collection on IGr(3,9)}}

\maketitle

\abstract{We construct a minimal Lefschetz decomposition of the bounded derived category of the odd isotropic Grassmannian $\IGr(3,9)$. The exceptional objects are $\Sp_9$-equivariant vector bundles. This provides further evidence of the extended Dubrovin conjecture \cite{kuznetsov2021residual}.}

\setcounter{tocdepth}{2}

\tableofcontents

\section{Introduction}\label{sec:intro}

The bounded derived category of coherent sheaves of a smooth projective variety $X$ is one of its most remarkable invariants. The derived category (from now on, denoted as $\DX$) can often be split in smaller triangulated subcategories. In the best case, these subcategories are generated by a single object with the easiest possible structure, an exceptional object.
The work of Kapranov \cite{kapranov1988derived} sparked the interest in the study of the derived category of homogeneous spaces by finding full exceptional collections on Grassmannians and quadrics. This led to the following conjecture. As in the rest of the work, we fix $\mathbb{C}$ as base field.
\begin{conj}\label{conj:homog}
If $\mathsf{G}$ is a semisimple algebraic group over $\mathbb{C}$ and $\mathsf{P}\subset \mathsf{G}$ is a parabolic subgroup, then there is a full exceptional collection in $\mathbf{D}^{\mathrm{b}}(\mathsf{G}/\mathsf{P})$.
\end{conj}

We refer to \cite[\textsection~1.1-\textsection~1.2]{kuznetsov2016exceptional} and \cite{grass} for a survey on this conjecture.  In addition, it was proved in \cite[Theorem~1.2]{kuznetsov2016exceptional} that any $\mathsf{G}/\mathsf{P}$ for $\mathsf{G}$ classical admits an exceptional collection of maximal possible length $r = \ranksf K_0(\mathsf{G}/\mathsf{P}) = \dimsf  \H^\bullet(\mathsf{G}/\mathsf{P},\mathbb{C})$, even though its fullness is not known. 

Dubrovin's conjecture \cite[Conjecture~4.2.2]{dubrovin1998geometry} claims that the existence of a full exceptional collection is equivalent to the generic semisimplicity of the big quantum cohomology ring. We consider the following version of the conjecture, which additionally involves Lefschetz collections. For a more detailed discussion, we refer to \cite[\textsection~1]{kuznetsov2021residual}. Recall that a rectangular Lefschetz collection with respect to a line bundle $\cL$ is given by an exceptional collection and its twists by powers of $\cL$  (\cf \cref{def:Lefschetz}). Recall that the \emph{index} $w\geq 0 $ of a Fano variety $X$ is the maximal positive integer such that the canonical class $K_X$ is divisible by $w$ in $\Pic X$. In that case, we write $\omega_X = \cO(-w)$,
where $\omega_X$ is the canonical bundle and $\cO(1)$ is a primitive ample line bundle on $X$. We denote the big and small quantum cohomology of $X$ by $\mathrm{BQH}(X)$ and $\mathrm{QH}(X)$. Recall that $\mathrm{QH}(X)$ is an algebra over the quantum parameters $\mathbb{Q}[q_1,\dots,q_s]$ corresponding to the functions on the affine space $\Pic X \otimes \mathbb{Q}$. Consider the algebra $\mathrm{QH}(X)_{can}=\mathrm{QH}(X)\otimes_{\mathbb{Q}[q_1,\dots,q_s]} \mathbb{C}$, where ${\mathbb{Q}[q_1,\dots,q_s]} \rightarrow \mathbb{C}$ is induced by $K_X\in \Pic X\otimes \mathbb{Q}$. Notice that $\mathrm{QH}(X)_{can} \cong \H^\bullet(X, \mathbb{C})$ as a vector space, but it is endowed with a different (quantum) multiplication.
\begin{conj}[{\cite[Conjecture~1.3.(i)]{kuznetsov2021residual}}]\label{conj:KS}
Let $X$ be a Fano variety with index $w$ and assume that  $\mathrm{BQH}(X)$ is generically semisimple.  If the class $[K_X]\in \H^2(X, \mathbb{C})\subset \mathrm{QH}(X)_{can} $ is invertible (with respect to the quantum multiplication), then there is an exceptional collection $E_1,\dots,E_p$ extending to a full rectangular Lefschetz collection of $\DX$, where $p=\frac{1}{w} \dimsf  \H^\bullet(X,\mathbb{C})$.
\end{conj}

A good testing ground for this conjecture is the class of \emph{horospherical varieties}, introduced by \cite{pasquier2008varietes}. These are normal algebraic varieties on which a reductive group acts with an open orbit isomorphic to a torus bundle over a homogeneous variety. Naturally, this includes  homogeneous spaces and toric varieties. Therefore, we believe that studying exceptional collections on horospherical varieties can shed light on exceptional collections in homogeneous spaces. 

Let $\mathsf{G}$ be a semisimple algebraic group over $\mathbb{C}$. Pasquier proved that except for homogeneous varieties, any smooth $\mathsf{G}$-horospherical variety of Picard rank one is a $\mathsf{G}$-variety which has exactly two disjoint closed orbits $Y, Z$ under the action of $\mathsf{G}$. The stabilizers of $Y$ and $Z$ are maximal parabolic subgroups $\mathsf{P}_Y,\mathsf{P}_Z \subset \mathsf{G}$. The following \cref{thm:pas} provides a classification of horospherical varieties. We write $\mathrm{Type}(\mathsf{G})$ for the Dynkin type of $\mathsf{G}$ and $\mathsf{P}_k$ for the maximal parabolic subgroup of $\mathsf{G}$ associated to the $k$-th fundamental weight  with respect to Bourbaki notation.
\begin{theorem}[{\cite[Theorem~0.1]{pasquier2009some}}]\label{thm:pas}
Let $X$ be a smooth projective $\mathsf{G}$-horospherical variety of Picard rank one. Then either $X$ is homogeneous,
or $X$ can be constructed from a triple $(\mathrm{Type}(\mathsf{G}), \mathsf{P}_Y, \mathsf{P}_Z)$ belonging to the following list:
\begin{itemize}
\setlength\itemsep{0em}
    \item[1.] $(\mathsf{B}_n, \mathsf{P}_{n}, \mathsf{P}_{n-1})$ with $n\geq 3$;
    \item[2.] $(\mathsf{B}_3, \mathsf{P}_3, \mathsf{P}_1)$;
    \item[3.] $(\mathsf{C}_n, \mathsf{P}_{k}, \mathsf{P}_{k-1})$ with $n\geq 2$ and $k\in \{2,\dots, n\}$;
    \item[4.] $(\mathsf{F}_4,\mathsf{P}_3,\mathsf{P}_2)$;
    \item[5.] $(\mathsf{G}_2,\mathsf{P}_2,\mathsf{P}_1)$.
\end{itemize}
\end{theorem}
An explicit construction of $X$ out of $(\mathrm{Type}(\mathsf{G}), \mathsf{P}_Y, \mathsf{P}_Z)$ can be found in \cite{pasquier2009some} and it is summarized in \cite[Proposition~1.6]{gonzales2018geometry}. As the homogeneous pieces $Y$ and $Z$ are enough to identify the horospherical variety $X$, it would be interesting to describe $\DX$ in terms of the homogeneous varieties $Y$, $Z$. As a first step towards such a description,  we prove in \cref{prop:invariants} that 
\begin{equation*}
    K_0(X)\cong K_0(Y) \oplus K_0(Z).
\end{equation*}

Following the computations of \cite{gonzales2018geometry}, we expect that if $X$ is a horospherical variety of Picard rank one, then $\DX$ admits a full exceptional collection and the big quantum cohomology ring $\mathrm{BQH}(X)$ is generically semisimple.

We summarize here the cases of horospherical varieties for which full exceptional collections are already known, grouped as in the classification of \cref{thm:pas}:
\begin{itemize}
    \item[\textbf{2.}] $(\mathsf{B}_3,\mathsf{P}_3,\mathsf{P}_1)$ by \cite[\textsection 6.2]{kuznetsov2006hyperplane};
    \item[\textbf{3.}] $(\mathsf{C}_n, \mathsf{P}_2, \mathsf{P}_1)$ by  \cite{pech2013quantum}, \cite{kuznetsov2008exceptional}; $(\mathsf{C}_3, \mathsf{P}_
    3, \mathsf{P}_2)$ by  \cite{fonarev2020bounded};
    \item[\textbf{5.}] $(\mathsf{G}_2,\mathsf{P}_2,\mathsf{P}_1)$ by \cite{gonzales2018geometry}.
\end{itemize}

The most interesting case in the list of \cref{thm:pas} is $(\mathsf{C}_n, \mathsf{P}_{k},\mathsf{P}_{k-1})$. The corresponding horospherical variety $X$ is the odd isotropic Grassmannian $\IGr(k,2n+1)$ of $k$-dimensional subspaces in a $(2n+1)$-dimensional space endowed with a skew-symmetric form $\psi$ of maximal possible rank $2n$. We refer to \cite{mihai2007odd} for a survey on the properties of these varieties. In \cite[Theorem~5.17]{gonzales2018geometry}, the authors provide a presentation of the small quantum cohomology ring and show its semisimplicity for $\IGr(2,2n+1)$ and $\IGr(3,7)$. Based on this description, Belmans could verify computationally (\cf \cite{grasshoro}) that the small quantum cohomology of $\IGr(k,2n+1)$ for $1\leq k\leq n\leq 7$ is generically semisimple.
Consequently, the extended Dubrovin's conjecture (\cref{conj:KS}) predicts that in all these cases $\IGr(k,2n+1)$ has a full exceptional collection. The main result of this paper, \cref{thm:final} stated below, proves this on $\IGr(3,9)$, the first case not covered by previous results. The prediction of \cref{conj:KS} in this case claims there should be an exceptional collection of $8$ elements that extends to a rectangular Lefschetz collection generating $\DX$.

To state our result we need to introduce some notation. Let $V$ be a $9$-dimensional vector space endowed with a skew-symmetric form $\psi$ of maximal rank $8$. Let us fix $X=\IGr(3,9)$, the isotropic Grassmannian of $3$-subspaces in $V$ with respect to $\psi$. Let $\cU$ be the tautological subbundle on $X$, recall that
\begin{equation*}
    \cO(1)=\wedge^3 \cU^*
\end{equation*}
is the ample generator of $\Pic X$. In this work, we construct a full rectangular Lefschetz collection of $\DX$ with respect to $\cO(1)$.

Given a $\GL_3$-dominant weight $\lambda$, we denote by $\cU^\lambda$ the bundle associated to the irreducible representation of $\GL_3$ of highest weight $\lambda$ and the frame bundle of $\cU^*$, so that
\begin{equation*}
    \cU^{m,0,0}= S^m \cU^*, \quad \cU^{0,0,-m}= S^m \cU, \quad \cU^{l,l,l}=\cO(l).
\end{equation*}
In other words, $\cU^\lambda$ is obtained by an application of the Schur functor associated to $\lambda$ to the vector bundle $\cU^*$

Consider the following collections of vector bundles on $\IGr(3,9)$:
\begin{equation*}
\begin{aligned}
    \bB_1 &= \{ \ \cU^{0,0,-2}, &\cU^{0,0,-1},\quad &\cU^{1,0,-1}, &\cU^{2,0,-1},\quad &\cU^{0,0,0},  &\cU^{1,0,0},\quad &\cU^{2,0,0} &\}, \\
    \bB_2 &= \{&\cU^{0,0,-1},\quad &\cU^{1,0,-1}, &\cU^{2,0,-1},\quad &\cU^{0,0,0},  &\cU^{1,0,0},\quad
    &\cU^{2,0,0}, &\cU^{3,0,0} \ \}.
\end{aligned}
\end{equation*}
Notice that $\bB_1$ and $\bB_2$ have $6$ elements in common, while their union has length $8$.
We check in \cref{cor:bases} that both $\bB_1$ and $\bB_2$ induce non full Lefschetz collections of length $7$. One could hope that $\bB_1\cup\bB_2$ is an exceptional collection, but this is not the case. In fact, by \cref{lem:extremal computation}, we have
\begin{equation*}
    \Ext^\bullet(\cU^{3,0,0}, \cU^{0,0,-2}) = \mathbb{C}[-4],\quad \quad \quad \Ext^\bullet(\cU^{0,0,-2}, \cU^{3,0,0} ) \neq 0,
\end{equation*}
so we cannot have both $\cU^{3,0,0}$ and $\cU^{0,0,-2}$ in the same exceptional collection. 

To solve this problem, we will replace the bundle $\cU^{3,0,0}\in \bB_2\setminus \bB_1$ by another object $\cH$, using a procedure analogous to the one used in \cite{guseva2020derived} and \cite{novikov}. More precisely,
we consider the following bicomplex:
\begin{equation}\label{H-intro}
\begin{tikzcd}
\wedge^3V^* \otimes \cO \ar[r] &\wedge^2V^* \otimes \cU^{1,0,0}\ar[r] &V^* \otimes \cU^{2,0,0} \ar[r]  &\cU^{3,0,0} 
\\
\wedge^2V^* \otimes \cU^{0,0,-1}\ar[r]\ar[u] &V^* \otimes \cU^{1,0,-1} \ar[r]\ar[u]  &\cU^{2,0,-1},\ar[u] &
\end{tikzcd}
\end{equation}
where the rows are given by the stupid truncations of the staircase complexes associated to $\cU^{3,0,0}$ and $\cU^{2,0,-1}$ (\cf \cref{thm:staircase}) and the vertical arrows are induced by the form $\psi$. The object $\cH$ is the totalization of this bicomplex.
See \eqref{def:H} and \cref{excep} for an alternative description of the object $\cH$. Denote by $\bB$ the following collection of bundles on $X$:
\begin{equation}\label{def:B-intro}
    \bB=\{\cH,\, \cU^{0,0,-2}, \, \cU^{0,0,-1}, \, \cU^{1,0,-1} ,\, \cU^{2,0,-1} ,\, \cU^{0,0,0} ,\, \cU^{1,0,0} ,\, \cU^{2,0,0} \}=\{\cH\}\cup \bB_1.
\end{equation} 

Our main result is the following.
\begin{theorem}\label{thm:final}
Let $X=\IGr(3,9)$. Then $\bB$ is an exceptional collection of $\Sp_9$-equivariant objects and it extends to a full rectangular Lefschetz collection  given by:
\begin{equation*}
    \DX=\langle \, \bB,\,\bB(1),\,\bB(2),\,\bB(3),\,\bB(4),\,\bB(5),\,\bB(6) \,\rangle.
\end{equation*}
\end{theorem}
The proof proceeds as follows. 
We prove that $\bB_1$ and $\bB_2$ defined before both induce Lefschetz bases of length $7$. Additionally, they satisfy $\Ext^\bullet(\bB_2(l),\bB_1)=0$ for $l=1,\dots, 6$.  Let  $\Bsc$ be the triangulated category generated by the objects in $\bB_1$ and $\bB_2$. Then, the collection of triangulated subcategories $\Bsc, \Bsc(1),\dots, \Bsc(6)$ is semiorthogonal. 

Using the staircase complexes associated to $\cU^{3,0,0}$ and $\cU^{2,0,-1}$, we prove that $\cH$ is the left mutation of $\cU^{3,0,0}$ through $\bB_1$, hence $\bB=\{\cH\}\cup \bB_1$  is a full exceptional collection of $\Bsc$ (actually, they are all vector bundles, \cf \cref{rem:vb}), proving that $\bB$ is a Lefschetz basis. 

To prove semiorthogonality we mainly rely on the embedding of $X$ in $\IGr(3,10)$, using the fact that $\IGr(3,10)$ is homogeneous under the action of $\Sp_{10}$. We prove a vanishing criterion for bundles on odd isotropic Grassmannians (\cref{prop:vanish-crit}) as application of Borel--Bott--Weil Theorem. To prove that $\bB$ induces a full Lefschetz collection, we follow the algorithmic method developed in \cite{novikov}.

\paragraph{Overview of the work.} In \cref{subsec:except,sec:prelim-homog} we cover the preliminaries and prove more general versions (\cref{prop:vanish-even} and \cref{prop:vanish-crit}) of cohomology vanishing lemmas previously established in \cite{fonarev2020bounded}. In \cref{sec:twocoll}, we prove that $\bB_1$ and $\bB_2$ are exceptional collections extending to (non full)  rectangular Lefschetz collections. In \cref{sec:H-props}, we show that the object $\cH$ is the left mutation of $\cU^{3,0,0}$ through $\bB_1$ and that $\bB$ is an exceptional collection. In \cref{sec:full}, we show the fullness of the rectangular Lefschetz collection induced by $\bB$.

\paragraph{Notation.} 
We fix $\mathbb{C}$ as base field. Unless otherwise noted, $X=\IGr(3,9)$ with $\DX$ its bounded derived category of coherent sheaves. 
We define the graded vector space $\Ext^\bullet(-,-)$ as $\oplus_i \Hom_{\DX}(-,-[i])[-i]$.

\paragraph{Acknowledgements.} I would like thank my advisors, Alexander Kuznetsov and Ugo Bruzzo, for several insightful conversations and their constant support. I was partially supported by the national research project PRIN 2017 "Moduli theory and birational geometry".

\section{Exceptional collections}\label{subsec:except}
In this section $X$ is an arbitrary smooth projective variety. 
A standard reference for this foundational material, except where otherwise noted, is \cite{bondal1990representable}. Let $\Asc\subseteq \DX$ be a full triangulated subcategory. We define the right and left orthogonal of $\Asc$ as: 
\begin{equation*}
    \Asc^\perp=\{E\in \DX\vert \Hom(\Asc,E)=0\}, \quad \quad \quad {}^{\perp}\Asc=\{E\in \DX\vert \Hom(E,\Asc)=0\}.
\end{equation*}

\begin{defn}
Let $\Asc\subseteq \DX$ be a full triangulated subcategory. Then $\Asc$ is called admissible if the embedding functor admits left and right adjoints.
\end{defn}

\begin{defn}
A sequence of full triangulated subcategories $\Asc_1,\dots,\Asc_p \subseteq \DX$ is called a semiorthogonal collection if $\Asc_i\subseteq \Asc^\perp_j$ for $1\leq i< j\leq p$. In that case, we denote the smallest triangulated subcategory containing all the $\Asc_i$ by $\langle\Asc_1,\dots,\Asc_p\rangle\subseteq \DX$. We say that a semiorthogonal collection $\Asc_1,\dots,\Asc_p$ is a decomposition of $\DX$ if $\DX=\langle\Asc_1,\dots,\Asc_p\rangle$ and every $\Asc_i$ is admissible.
\end{defn}
If $\Asc\subseteq \DX$ is admissible, then the following:
\begin{equation*}
    \DX=\langle \Asc^\perp,\Asc \rangle, \quad \quad \quad \DX=\langle \Asc, {}^{\perp}\Asc \rangle,
\end{equation*}
are semiorthogonal decompositions.
\begin{defn}
An object $E\in \DX$ is exceptional if $\Ext^\bullet(E,E) = \mathbb{C}$.

An exceptional collection is a sequence of exceptional objects $E_1,\dots,E_p\in \DX$ with $E_i\in E_j^\perp$ for $1\leq i< j\leq p$. An exceptional collection is said to be full if $\DX=\langle E_1,\dots,E_p\rangle$.
\end{defn}
Recall that every subcategory $\Asc\subseteq\DX$ generated by an exceptional collection is admissible, \cf
\cite[Theorem~2.10]{bondal1990representable}.

The notion of rectangular Lefschetz decomposition is a natural way of studying decompositions of $\DX$ knowing a line bundle on $X$. Recall that the \emph{index} $w\geq 0 $ of a  Fano variety $X$ is the maximal positive integer such that $\omega_X = \cO(-w)$,
where $\omega_X$ is the canonical bundle and $\cO(1)$ is a primitive ample line bundle on $X$. 

\begin{defn}[{\cite[Definition~4.1]{kuznetsov2007homological}}]\label{def:Lefschetz}
A rectangular Lefschetz exceptional collection of $\DX$ with respect to $\cO(1)$ is an exceptional collection of the form:
\begin{equation*}
    \langle\, E_1,\dots,E_p, E_1(1),\dots,E_p(1),\dots, E_1(w-1),\dots,E_p(w-1)\,\rangle \subseteq \DX.
\end{equation*}
We say that $ E_1,\dots,E_p$ is the basis of the rectangular Lefschetz collection and $p$ is its length.
\end{defn}

Recall that for any $0\neq F\in \DX$, we have by Serre duality
\begin{equation*}
    \Ext^\bullet(F(w),F)= \Ext^\bullet(F,F)^*[-\dimsf  X]\neq 0;
\end{equation*}
therefore, $w$ is the maximal number of twists of $E_1,\dots,E_p$ which can be semiorthogonal. 

We state here for completeness a criterion to verify that an exceptional collection extends to a rectangular Lefschetz collection.

\begin{lemma}[{\cite[Lemma~2.18]{fonarev2020bounded}}]\label{lem:criterion-lef} 
A collection of objects $E_1,\dots,E_p \in \DX$ is a basis of a rectangular Lefschetz collection if and only if 
\begin{itemize}
    \item $E_1,\dots,E_p$ is an exceptional collection, and
    \item $\Ext^\bullet(E_j(t), E_i) = 0$ for $1\leq i \leq j \leq p$ and all $1\leq t\leq w-1$,
\end{itemize}
where $w$ is the index of $X$.
Let $\Asc\subseteq \DX $ be the smallest full triangulated subcategory containing $E_1,\dots,E_p$. If only the second condition holds, then ${\Asc, \Asc(1),\dots, \Asc(w-1)}$
is only a semiorthogonal collection.
\end{lemma}
\begin{proof}
    The second hypothesis proves one half of the required semiorthogonality conditions; that is $\Ext^\bullet(E_j(t), E_i) = 0$ for $j\geq i$. The other half follows by Serre duality. This proves the second claim. 
    
    Taking the first hypothesis into account, we also obtain the first claim.
\end{proof}

We now introduce the mutation functors with respect to an admissible subcategory $\Asc$. For any object $F\in \DX$ there are unique and functorial triangles: 
\begin{equation*}
    \mathbb{R}_{\Asc} F\rightarrow F\rightarrow F'  \quad \textrm{and} \quad F''\rightarrow F \rightarrow \mathbb{L}_{\Asc} F ,
\end{equation*}
where $F', F''\in \Asc$ while $\mathbb{R}_{\Asc} F\in {}^{\perp}\Asc$ and $\mathbb{L}_{\Asc} F \in \Asc^\perp$. Both $\mathbb{R}_{\Asc}$ and $\mathbb{L}_{\Asc}$ vanish on $\Asc$ and induce mutually inverse equivalences between ${}^{\perp}\Asc$ and $\Asc^{\perp}$. 

The mutation functors $\mathbb{R}_{\Asc}$ and $\mathbb{L}_{\Asc}$ take a more explicit form when $\Asc=\langle \,E_1,\dots, E_p\,\rangle$, where $E_1,\dots, E_p$ is an exceptional collection. It is immediate to verify:
\begin{equation*}
    \mathbb{R}_{\langle\, E_1,\dots, E_p \, \rangle} = \mathbb{R}_{E_p} \circ \dots \circ \mathbb{R}_{E_1}  \quad \textrm{and} \quad \mathbb{L}_{\langle \, E_1,\dots, E_p \, \rangle} = \mathbb{L}_{E_1} \circ \dots \circ \mathbb{L}_{E_p}.
\end{equation*}
Moreover, for any exceptional object $E$ we can write: 
\begin{equation*}
    \mathbb{R}_E F = \Cone{}(\,F\rightarrow \Ext^\bullet(F,E)^*\otimes E \,)[-1] \quad \textrm{and} \quad \mathbb{L}_E F = \Cone{}(\,\Ext^\bullet(E,F)\otimes E \rightarrow F\,),
\end{equation*}
where the morphisms are the canonical coevaluation and evaluation maps. 

\section{Cohomology on isotropic Grassmannians}\label{sec:prelim-homog}
\subsection{Grassmannians} 
Let $V$ be an $m$-dimensional vector space and let $\GL(V)$ be the corresponding linear group. For $1\leq k\leq m-1$, we denote the Grassmannian of $k$-planes in $V$ by $\Gr(k,V)$.
We recall the tautological exact sequence of bundles on $\Gr(k,V)$:
\begin{equation}\label{eq:taut-gl}
    0\rightarrow \cU \rightarrow V\otimes \cO \rightarrow \cQ \rightarrow 0,
\end{equation}
where $\cU$ is the tautological subbundle of rank $k$ and $\cQ$ is the tautological quotient bundle of rank $m-k$ . Note that the canonical sheaf satisfies:
\begin{equation*}
    \omega_{\Gr(k,V)} = \cO(-m),
\end{equation*}
where $\cO(1)= \det \cU^*$ is the ample generator of the Picard group. We use the following notation:
\begin{equation}\label{eq:Uperp}
    \cU^{\perp} = \cQ^*.
\end{equation}
In the coming sections, we discuss the preliminaries that will be needed in the computations on $\bD^b(\Gr(k,m))$: Schur functors, the Littlewood--Richardson rule and Koszul and staircase complexes. The only new result in this section is \cref{prop:split-staircase}, where we describe the restriction of a staircase complex to a smaller Grassmannian.

\subsubsection{Schur functors and Littlewood--Richardson rule} \label{ssub:BWB-GL}
We summarize here some basic facts about the representation theory of $\GL_k$ and the associated Schur functors. The weight lattice of $\GL_k$ is isomorphic to $\mathbb{Z}^k$, and its subset
\begin{equation*}
    P_k^+=\{\,\lambda\in \mathbb{Z}^k\mid \lambda_1\geq \lambda_2\geq \dots \geq \lambda_k\,\}
\end{equation*}
is the cone of dominant weights. 
Given $\lambda\in P_k^+$, we denote by $V_{\GL}^\lambda$ the $\GL(V)$-representation with highest weight $\lambda$. In particular, the symmetric and wedge powers are:
\begin{equation*}
 V_{\GL}^{p,0,\dots,0} = S^p V^*, \quad\quad\quad
    V_{\GL}^{1,\dots,1,0\dots,0} = \wedge^p V^*.
\end{equation*}
If $\lambda \in P_k^+$, we denote:
\begin{equation*}
    -\lambda=(-\lambda_k,-\lambda_{k-1},\dots, -\lambda_1) \quad \quad \textrm{and} \quad \quad \abs{\lambda}=\sum \lambda_i. 
\end{equation*}
We recall that: 
\begin{equation*}
    V_{\GL}^{-\lambda}= (V_{\GL}^\lambda)^* = (V_{\GL}^*)^{\lambda}.
\end{equation*}
There is a natural partial ordering on $P_k^+$ given by
\begin{equation}\label{eq:order}
    \mu \subseteq \lambda \Leftrightarrow \mu_i \leq \lambda_i \quad  \text{for all}
    \quad i=1,\dots,k.
\end{equation} 

Let $\cE$ be a vector bundle of rank $k$ on an algebraic variety and consider the principal $\GL_k$-bundle of its frames. Let $\cE^\lambda$ be the vector bundle associated to the representation $V_{\GL}^\lambda$. 
Once we fixed $\lambda\in P^+_k$, the functor $ \mathbf{VB}_k(X)\rightarrow \mathbf{VB}(X)$,  $\cE \mapsto \cE^\lambda$ defined as above is the \textit{Schur functor} associated to $\lambda$.
In particular, the symmetric and wedge powers of a vector bundle are Schur functors:
\begin{equation}\label{eq:Schur-powers}
     \cE^{p,0,\dots,0} = S^p\cE^* , \quad\quad\quad
    \cE^{1,\dots,1,0\dots,0} = \wedge^p \cE^*.
\end{equation}
It also follows that:
\begin{equation}\label{eq:weight-dual}
    \cE^{-\lambda}=(\cE^\lambda)^*=(\cE^*)^\lambda.
\end{equation} 

The \emph{Littlewood--Richardson rule} allows one to decompose the representation $V_{\GL}^\alpha \otimes V_{\GL}^\beta$ as a direct sum of representations $V_{\GL}^\gamma$, possibly with multiplicities. Considering their associated Schur functors, we obtain an induced $\GL$-equivariant decomposition of $\cE^\alpha \otimes \cE^\beta$ in $\cE^\gamma$. We will often use the notation 
\begin{equation*}
    \cE^\gamma \inplus \cE^\alpha \otimes \cE^\beta
\end{equation*}
to say that $\cE^\gamma$ is a direct summand of the right hand side. 
A statement of the Littlewood--Richardson rule is out of the scope of this work, for that we refer to \cite[Theorem~2.3.4]{weyman2003cohomology}. Instead, we describe some simpler special cases.

\begin{proposition}[Pieri's Formulas, {\cite[Corollary~2.3.5]{weyman2003cohomology}}]\label{prop:pieri}
Let $\lambda\in P_k^+$ and let $j$ be a positive integer. Then, there are direct sum decompositions
\begin{equation*}
    \cE^\lambda \otimes S^j \cE^* = \bigoplus_{\gamma \in \HS_\lambda^j} \cE^\gamma \quad \text{ and } \quad \cE^\lambda \otimes \wedge^j \cE^*  = \bigoplus_{\gamma \in \VS_\lambda^j} \cE^\gamma
\end{equation*}
where 
\begin{align*}
    \HS_\lambda^j &= \{\gamma \in P_k^+ \; \text{with} \;  \abs{\gamma} - \abs{\lambda}=j \; \text{and} \; \gamma_1\geq \lambda_1 \geq \gamma_2 \geq \lambda_2 \geq \dots \geq \lambda_k \geq \gamma_k\}, \\
    \VS_\lambda^j &= \{\gamma \in P_k^+  \; \text{with} \; \abs{\gamma} - \abs{\lambda}=j \; \text{and} \; \lambda_i + 1\geq \gamma_i \geq \lambda_i \; \text{for all} \; i=1,\dots, k\}.
\end{align*}
\end{proposition}

In \cref{prop:pieri} above, the notation $\HS$ and $\VS$ comes from the representation of weights as Young diagrams. Indeed, if $\gamma \in \HS_\lambda^j$, then $\gamma\setminus \lambda$ is a \emph{horizontal strip} ($j$ boxes, at most one box per column), while if $\gamma \in \VS_\lambda^j$, then $\gamma\setminus \lambda$ is a \emph{vertical strip} ($j$ boxes, at most one box per row).
The following corollary is an immediate application of Pieri's formulas. 
\begin{corollary}\label{cor:LR-det}
Let $V$ be a vector space of rank $k$. Alternatively, let $\cE$ be a vector bundle of rank $k$. Let $\lambda\in P_k^+$ and let $l$ be an integer. Then:
\begin{equation*}
    V_{\GL}^\lambda \otimes V_{\GL}^{(l,l,\dots,l)}=V_{\GL}^{\lambda+(l,l,\dots,l)},\quad \quad \cE^\lambda \otimes \cE^{(l,l,\dots,l)}=\cE^{\lambda+(l,l,\dots,l)},
\end{equation*}
where $(l,l,\dots,l)$ denotes the vector of $P^+_k$ with $k$ entries equal to $l$.
\end{corollary}

During the course of the work, we will have to study tensor products of representations where neither weights are elementary enough to apply Pieri's formulas. In those cases, we will use the following lemmas. 

\begin{lemma}[{\cite[Proposition~2.3.1]{weyman2003cohomology}, \cite[Lemma~3.3]{guseva2020derived}}]\label{lem:sum}
Let $\alpha,\beta \in P_k^+$. Suppose
\begin{equation*}
    \cE^\gamma \inplus \cE^{-\alpha}\otimes \cE^{\beta}.
\end{equation*}
Then 
\begin{equation*}
    \beta_k-\alpha_{k+1-i}\leq \gamma_i \leq \beta_i-\alpha_k \quad \textrm{for all} \quad 1\leq i\leq k \quad \textrm{and} \quad \abs{ \gamma}=\abs{\beta}-\abs{\alpha}.
\end{equation*}
As a straightforward consequence:
\begin{equation}
\label{eq:jumps}
    \gamma_i-\gamma_{i+1} \leq (\beta_i - \beta_k)+ (\alpha_{k-i} - \alpha_k) \quad \textrm{for all} \quad 1\leq i< k.
\end{equation}
\end{lemma}

\begin{lemma}\label{lem:Schur}

Let $\alpha, \beta \in P_k^+$, and let $l\in \mathbb{Z}$. Let
\begin{equation*}
    \cE^{-\alpha}\otimes \cE^{\beta} = \bigoplus_\gamma \cE^\gamma
\end{equation*}
be the Littlewood--Richardson decomposition. Then the following statements are equivalent:
\begin{enumerate}
    \item\label{it:eq1} $\beta - \alpha  = (l,\dots,l)$,
    \item\label{it:eq2} the weight $(l,\dots,l)$ appears once and only once among $\gamma$.
    \item\label{it:eq3} the weight $(l,\dots,l)$ appears among $\gamma$.
\end{enumerate}
\end{lemma}
\begin{proof}
By Schur's Lemma, $\cE^{(l,\dots,l)}$ is a direct summand of the Littlewood--Richardson decomposition of $\cE^{-\alpha}\otimes \cE^{\beta}$ if and only if 
\begin{equation*}
\Hom(V_{\GL}^{(l,\dots,l)}, V_{\GL}^{-\alpha}\otimes V_{\GL}^{\beta}) \neq 0.
\end{equation*}
On the other hand:
\begin{equation*}
    \Hom(V_{\GL}^{(l,\dots,l)}, V_{\GL}^{-\alpha}\otimes V_{\GL}^{\beta}) =\Hom(V_{\GL}^{(l,\dots,l)}\otimes V_{\GL}^{\alpha}, V_{\GL}^{\beta}) =  \Hom(V_{\GL}^{\alpha + (l,\dots,l)}, V_{\GL}^{\beta}),
\end{equation*}
where the first equality holds by \eqref{eq:weight-dual} and the second is an application of \cref{cor:LR-det}.
The last $\Hom$-space being nonzero is equivalent to $\alpha + (l,\dots,l) = \beta$. As $\dimsf  \Hom(V_{\GL}^{\alpha + (l,\dots,l)}, V_{\GL}^{\beta})$ equals the multiplicity of $V_{\GL}^\beta$ in $V_{\GL}^{\alpha + (l,\dots,l)}$, if it is nonzero, it must be $1$.
\end{proof}

\subsubsection{Koszul and staircase complexes}\label{sec:koszul}

We recall here some exact sequences on $\Gr(k,m)$. For each $p$, the tautological sequence \eqref{eq:taut-gl} induces the following exact sequence, which is called the \emph{Koszul complex}:
\begin{equation}\label{eq:Koszul-Symm}
    0\rightarrow \wedge^p \cU^\perp \rightarrow \wedge^p V^*\otimes \cO \rightarrow \dots \rightarrow V^*\otimes  S^{p-1} \cU^*\rightarrow S^p \cU^* \rightarrow 0.
\end{equation}
Applying Borel--Bott--Weil Theorem (see \cite{demazure1976very}), it is immediate to see that the differentials in \eqref{eq:Koszul-Symm} are the unique nonzero $\GL(V)$-equivariant maps between the terms in the complex.

We will now introduce the second and most important family of exact sequences for this work, \emph{staircase complexes}, 
introduced in \cite{fonarev2013minimal}. First we recall the general theory. 
Let $\lambda=(\lambda_1,\lambda_2,\dots, \lambda_k)\in P_k^+$, with $\lambda_1 = m-k$ and $\lambda_k\geq 0$. We define
\begin{equation*}
    \lambda'=(\lambda_2,\dots, \lambda_k, 0)\in P_k^+.
\end{equation*}

\begin{theorem}[Staircase complex,\,{\cite[Proposition 5.3]{fonarev2013minimal}}]\label{thm:staircase}
Let $\lambda=(\lambda_1,\lambda_2,\dots, \lambda_k)\in P_k^+$ with $\lambda_1 = m-k$ and $\lambda_k\geq 0$. Then there exist a sequence of weights $\{\mu_i\}_{i\in\{1,\dots, m-k\}}\in P_k^+$ and an exact sequence:
\begin{equation}\label{eq:staircase}
    0\rightarrow \cU^{\lambda'}(-1)\xrightarrow {\gamma_{m-k+1}} \wedge^{\nu_{m-k}}V^*\otimes\cU^{\mu_{m-k}}\xrightarrow {\gamma_{m-k}} \dots \xrightarrow {\gamma_{2}} \wedge^{\nu_1}V^*\otimes \cU^{\mu_1}\xrightarrow {\gamma_1}\cU^{\lambda}\rightarrow 0, 
\end{equation} 
where $\nu_i = \abs{\lambda}-\abs{\mu_i}$ and the morphisms $\{\gamma_i\}_{i\in \{1,\dots,m-k+1\}}$ are the unique nonzero $\GL(V)$-equivariant maps between the terms of the complex. The sequence $\{\mu_i\}_{i\in\{1,\dots, m-k\}}$ is totally ordered by inclusion; moreover, the weights $\mu_i$  are all positive and are contained in $\lambda$. 
\end{theorem}

We do not report the construction of the weights $\mu_i$, except in the case $k=3$ (see \eqref{eq:weights} below), which is the only case used in the body of the paper. Note however that if $\lambda_1>\lambda_2$, then $\mu_1=(\lambda_1-1, \lambda_2, \dots, \lambda_k)$.

\begin{remark}\label{rem:other-stair}
Notice that the conditions $\lambda_1=m-k$ and $\lambda_k\geq 0$ can be achieved for any $\lambda\in P_k^+$ with $\lambda_1-\lambda_k\leq m-k$ by twisting appropriately. We will still refer to this complex as the staircase complex associated to $\cU^{\lambda}$. With the notation of \cref{thm:staircase}, we denote both the acyclic complex defined above and the one obtained in \cref{rem:other-stair} as $\Stair(\cU^{\lambda})$. 
\end{remark}

We now consider the problem of the restriction of a staircase complex on $\Gr(k,m+1)$ to $\Gr(k,m)$. Let $V$ be a $m$-dimensional vector space. Fix a $(m+1)$-dimensional vector space $\wV$ such that $V\subset \widetilde{V}$ is a hyperplane, then we obtain an induced embedding $j:\Gr(k, V)\rightarrow \Gr(k, \widetilde{V})$. Let $\widetilde{\cU}$ be the tautological bundle on $ \Gr(k, \widetilde{V})$. Recall that: 
\begin{equation}\label{eq:repeat}
    j^*(\widetilde{V}\otimes \cO)\cong (V\otimes \cO)\oplus \cO,  \quad j^*\widetilde{\cU}^\lambda \cong \cU^\lambda.
\end{equation}  
for every $\lambda\in P^+_k$.  We prove our result.
\begin{proposition}\label{prop:split-staircase}
    Let $\lambda\in P^+_k$ with $\lambda_1 = m+1-k > \lambda_2$ and $\lambda_k\geq 1$. Let $\Stair(\widetilde{\cU}^{\lambda})$ be the associated staircase complex on $\Gr(k,m+1)$. Then the restriction of the complex to $\Gr(k,m)$ splits:
    \begin{equation*}
        j^*\Stair(\widetilde{\cU}^{\lambda})\cong \Stair(\cU^{\lambda}) \oplus \Stair(\cU^{\mu_1})[1].
    \end{equation*}
\end{proposition}
\begin{proof}
   By \cite[Lemma~5.1]{fonarev2013minimal},  $\Stair(\widetilde{\cU}^{\lambda})$ can be characterized in the following way:
    \begin{equation}\label{eq:eps-repr}
        \varepsilon\in \Ext^{m+1-k}(\widetilde{\cU}^{\lambda}, \, \widetilde{\cU}^{\lambda'}(-1))\cong\H^{m+1-k}(\Gr(k,m+1),\, \widetilde{\cU}^{-1,\dots,-1,-(m+1-k)-1}\,)=\mathbb{C},
    \end{equation}
    where $\varepsilon$ is the unique nonzero extension. By \cref{thm:staircase}, the central part of $\Stair(\wcU^{\lambda})$ is obtained by decomposing the complex 
    \begin{equation}\label{eq:gr-cone}
        {\Cone(\, \widetilde{\cU}^{\lambda}\,\xrightarrow \varepsilon \,\widetilde{\cU}^{\lambda'}(-1)[m+1-k]\,)}.
    \end{equation} 
    with respect to the exceptional collection on $\Gr(k,m+1)$ given by $\langle \, \widetilde{\cU}^{\gamma}\, \mid \, 0\subseteq \gamma \subseteq \mu_1 \subset \lambda\, \, \rangle$ (see \cite[Theorem 2.1]{fonarev2013minimal}).  Since $\lambda_1> \lambda_2$, we have $\mu_1 = (m-k,\,\lambda_2,\,\dots,\,\lambda_k)$; hence, the restriction of this collection to $\Gr(k,m)$:
    \begin{equation}\label{eq:sod-Gr(k,m)}
     \langle \, 
    {\cU}^\gamma\, \mid \, 0\subseteq \gamma \subseteq \mu_1 \, \rangle=\langle \, j^*\widetilde{\cU}^{\gamma}\, \mid \,  0\subseteq \gamma \subseteq \mu_1  \, \rangle ,   
    \end{equation}
    is an exceptional collection as well, by \cite[Theorem 2.1]{fonarev2013minimal}. 

     We now prove that $j^*\varepsilon = 0$ on $\Gr(k,m)$. 
    Similarly to the proof of \cite[Lemma~5.1]{fonarev2013minimal}, we apply Borel--Bott--Weil Theorem on $\Gr(k,m)$ (see \cite{demazure1976very}) obtaining: 
    \begin{equation}\label{eq:cohom}
        \H^p(\Gr(k,m), \, \cU^{-1,\dots,-1,-(m+1-k)-1}\,)=\begin{cases*}
            V\otimes \wedge^m V & if $p=m-k$,\\
            0 & otherwise.
        \end{cases*}
    \end{equation}
    As $j^*\varepsilon\in \H^{m+1-k}(\Gr(k,m), \, j^*{\wcU}^{-1,\dots,-1,-(m+1-k)-1}\,)=0$ by \eqref{eq:eps-repr} and \eqref{eq:cohom}, we obtain:
    \begin{equation*}
j^*\Cone(\varepsilon)=\Cone(j^*\varepsilon)=\Cone(0)=\cU^{\lambda}[1]\oplus \cU^{\lambda'}(-1)[m-k].
    \end{equation*} 
     
    Now, consider the left resolution of $j^*\widetilde{\cU}^{\lambda}=\cU^{\lambda}$ given by the staircase complex on $\Gr(k,m)$ (\cf \cref{rem:other-stair}). 
    We introduce the following weights: 
    \begin{equation*}
        \overline{\lambda}=(\lambda_1-1,\dots,\lambda_{k}-1)\in P_k^+ \quad \quad \textrm{and} \quad \quad  \overline{\lambda}'=(\lambda_2-1,\dots,\lambda_k-1,0)\in P_k^+.
    \end{equation*}
    Notice that $\overline{\lambda}$ satisfies the conditions in \cref{thm:staircase}.
    As a consequence, we find that 
    \begin{equation*}
        \cU^{\lambda}=\cU^{\overline{\lambda}}(1)\in \langle \, \cU^{\overline{\lambda}'},\,
    {\cU}^\gamma\, \mid \, (1,\dots,1)  \subseteq \gamma \subseteq \mu_1  \, \rangle \subseteq \langle \, 
    {\cU}^\gamma\, \mid \, 0\subseteq \gamma \subseteq \mu_1 \, \rangle,
    \end{equation*}
    where the rightmost term is the exceptional collection in \eqref{eq:sod-Gr(k,m)}.
    
    On the other hand, $\cU^{\lambda'}(-1)$ belongs to \eqref{eq:sod-Gr(k,m)} because it is the leftmost term of $\Stair(\cU^{\mu_1})$ in $\Gr(k,m)$, as $\mu_1=(m-k, \lambda_2, \dots, \lambda_k)$. As a consequence, we can recover $\Stair(\cU^{\mu_1})$ as the decomposition of $\cU^{\lambda'}(-1)$ with respect to \eqref{eq:sod-Gr(k,m)}. Taking the direct sum of these resolutions finally provides a decomposition of $\cU^{\lambda}[1]\oplus \cU^{\lambda'}(-1)[m-k]$ in terms of \eqref{eq:sod-Gr(k,m)}.

    By construction, the central truncation of $j^*\Stair(\wcU^{\lambda})$ gives a resolution  of $j^*\Cone(\varepsilon)$  with respect to \eqref{eq:sod-Gr(k,m)}. By uniqueness of the decomposition, it must agree with the decomposition given by the direct sum of the truncated staircase complexes. This proves the claim.
    \end{proof}

We restate here the description of a staircase complex on $\Gr(3,m)$. In \cref{thm:staircase}, fix any $\GL_3$-dominant weight $\lambda=(a,0,-b)$ such that $a+b\leq m-3$. Then the following collection of $m-1$ weights $\{\mu_i\}_{i\in\{0,\dots, m-2\}}$ of $\GL_3$
\begin{equation}\label{eq:weights}
    \mu_i :=\begin{cases*}
    (a-i,0,-b),  & \textrm{if} \quad $0\leq i\leq a;$\\
    (-1,a-i,-b),  & \textrm{if} \quad $a< i\leq a+b;$\\
    (-1,-b-1,a-i), & \textrm{if} \quad
    $a+b< i\leq m-2$.
    \end{cases*}
\end{equation}
are the weights appearing in the statement; note that $\mu_0=\lambda$.

\subsection{Even isotropic Grassmannian}

Let $W$ be a $2n$-dimensional vector space endowed with a symplectic form $\psi$ and let $\Sp(W)$ be the corresponding symplectic group. Note that $\psi$ induces a $\Sp(W)$-equivariant isomorphism $W\cong W^*$, which we will use quite often. 

For $1\leq k\leq n$, we denote the isotropic Grassmannian of $k$-planes in $W$ by $\IGr(k,W)$. We often refer to $\IGr(k,W)$ as the \textit{even isotropic Grassmannian} to emphasize the difference with the \textit{odd isotropic Grassmannian} defined in \cref{subsec:odd-grass}.

The variety $\IGr(k,W)$ is isomorphic to the zero locus of
\begin{equation*}
    \psi \in \wedge^2 W^* = \H^0(\Gr(k,W),\wedge^2 \cU^*).
\end{equation*}
The tautological sequence on $\IGr(k,W)$ is the restriction of the one in $\Gr(k,W)$, i.e.:
\begin{equation}\label{eq:taut}
    0\rightarrow \cU\rightarrow W\otimes \cO \rightarrow \cQ \rightarrow 0.
\end{equation}
The symplectic form induces a canonical embedding $\cU\hookrightarrow \cU^{\perp}$ (\cf \eqref{eq:Uperp}), so that we can define the quotient bundle $\cS$:
\begin{equation}\label{eq:S-def}
    0\rightarrow\cU\rightarrow \cU^\perp\rightarrow \cS\rightarrow 0.
\end{equation}
Additionally, $\cS$ is also the cohomology sheaf in the middle term of the complex:
\begin{equation}\label{eq:symplectic}
    0\rightarrow\cU\rightarrow W\otimes \cO\xrightarrow \psi \cU^* \rightarrow 0.
\end{equation}
Notice that $\cS$ is a bundle of $\ranksf \cS =  2(n-k)$, endowed with a symplectic isomorphism $\cS\cong \cS^*$ induced by $\psi$. The following result is well known, \cf \cite[Lemma~2.19, Proposition~9.7]{kuznetsov2016exceptional}.

\begin{lemma}\label{lem:even-inv}
The dimension $d$ and the index $w$ of $\IGr(k,2n)$ are respectively: 
\begin{equation*}
    d=\frac{k(4n-3k+1)}{2}, \quad w = 2n+1-k.
\end{equation*}
Moreover, the Grothendieck group is a free abelian group and its rank $r$ is:
\begin{equation*}
    r = \binom{n}{k}\cdot 2^k.
\end{equation*}
\end{lemma}

In the rest of the section, we state Borel--Bott--Weil Theorem on the even isotropic Grassmannian and we apply it to obtain a cohomology vanishing result.

\subsubsection{Borel--Bott--Weil Theorem}\label{sec:BWB}
The weight lattice of $\Sp_{2n}$ is isomorphic to $\mathbb{Z}^n$, and its subset
\begin{equation*}
    T_n^+= \{\lambda\in \mathbb{Z}^n\mid \lambda_1\geq \lambda_2\geq \dots \geq \lambda_n\geq 0\}
\end{equation*}
is the cone of dominant weights. Given $\lambda\in T_n^+$, we denote by $W^\lambda_{\Sp}$ the $\Sp(W)$-representation of highest weight $\lambda$. Our convention is such that:
\begin{equation*}
W_{\Sp}^{m,0,\dots,0} = S^m W^* \cong S^m W.
\end{equation*} The Weyl group of $\Sp_{2n}$ is the semidirect product $\bS_n\ltimes(\mathbb{Z}/2\mathbb{Z})^n$ acting on $\mathbb{Z}^n$, where $\bS_n$ acts by permutation and $(\mathbb{Z}/2\mathbb{Z})^n$ acts by changing signs of the coordinates. Let $\ell:\bS_n\ltimes(\mathbb{Z}/2\mathbb{Z})^n\rightarrow \mathbb{Z}$ be the length function. Thus, $\ell(\sigma)$ is the smallest number of simple reflections required to decompose $\sigma$. In this case, simple reflections are of two types: transposition of adjacent terms and the change of sign $(\lambda_1,\dots,\lambda_{n-1},\lambda_n)\mapsto (\lambda_1,\dots,\lambda_{n-1},-\lambda_n)$.  
We recall that the sum of fundamental weights of $\Sp_{2n}$ is
\begin{equation*}
    \rho_{\Sp_{2n}}=(n,n-1,\dots,1).
\end{equation*}

We recall here an immediate application of the more general Borel--Bott--Weil theorem, which will cover most of the computations in this work.
\begin{theorem}[Borel--Bott--Weil Theorem, \cite{demazure1976very}]\label{thm:BWB}
Let $\lambda\in P^+_k$. Let
\begin{equation*}
    \gamma=(\gamma_1,\dots,\gamma_n):=(\lambda_1,\dots, \lambda_k,0,\dots, 0)\in \mathbb{Z}^n,
\end{equation*} 
be the extension of $\lambda$ by $n-k$ zeros. Suppose that $\gamma+\rho_{\Sp_{2n}}$ has a zero entry or two entries with same absolute value, then:
\begin{equation*}
    \H^\bullet (\IGr(k,W), \cU^\lambda)=0.
\end{equation*}
Otherwise, let $\sigma\in \bS_n\ltimes(\mathbb{Z}/2\mathbb{Z})^n$ be the unique element of the Weyl group such that
\begin{equation*}
\gamma'=\sigma(\gamma+\rho_{\Sp_{2n}})-\rho_{\Sp_{2n}}\in T^+_n
\end{equation*}
is $\Sp_{2n}$-dominant, then we have:
\begin{equation*}
    \H^\bullet (\IGr(k,W), \cU^\lambda)=  W_{\Sp}^{\gamma'}[-\ell(\sigma)].
\end{equation*}
\end{theorem}

\subsubsection{Vanishing}

The coming result and the following \cref{prop:vanish-crit} are a straightforward generalization of the ideas presented in \cite[\textsection~4.1]{fonarev2020bounded}, where they were proved only in the case $k=n$. For further convenience, we state the following result for $\IGr(k,2n+2)$ instead of $\IGr(k,2n)$. 

\begin{proposition}\label{prop:vanish-even}
Let $\lambda=(\lambda_1,\dots,\lambda_k)\in P_k^+$ be a dominant weight of the group $\GL_k$ such that
\begin{enumerate}
    \item\label{it:ve-1} $\lambda_k < 0$,
    \item\label{it:ve-2} $\lambda_1 \geq -2(n+1) + k$,
    \item\label{it:ve-3} $\lambda_i-\lambda_{i+1} \leq 2(n+2-k)-1$ for $i=1,\dots,k-1$.
\end{enumerate}
Then we have:
\begin{equation*}
    \H^\bullet(\IGr(k, 2n+2), \cU^\lambda)=0,
\end{equation*}
that is, the bundle $ \cU^\lambda$ on $\IGr(k, 2n+2)$ is acyclic.
\end{proposition}

\begin{proof}
Suppose that $\cU^\lambda$ is not acyclic. Then, according to \cref{thm:BWB}, the absolute values of the entries in the sequence
\begin{equation*}
\gamma = (n+1+\lambda_1, n+\lambda_2,\dots, n+2-k+\lambda_k,n+1-k,\dots,1)    
\end{equation*}
have distinct nonzero absolute values. In particular, the entries $\gamma_i$, $1\leq i \leq k$, satisfy
\begin{equation}\label{eq:van-even}
    \gamma_i\geq n+2-k \quad \textrm{or} \quad \gamma_i\leq -(n-k+2).
\end{equation}
Using conditions \ref{it:ve-1} and \ref{it:ve-2} we have: 
\begin{align*}
    \gamma_1 = n+1+\lambda_1 \geq - (n+1-k) \quad \textrm{and} \quad
    \gamma_k = n+2-k+\lambda_k\leq n+1-k,
\end{align*}
which, together with \eqref{eq:van-even} imply
\begin{equation*}
    \gamma_1 \geq n+2-k \quad \textrm{and} \quad \gamma_k \leq -(n+2-k).
\end{equation*}
As $\lambda$ is dominant, the first $k$ entries of $\gamma$ are strictly decreasing, hence there must be some $j\in \{1,\dots, k-1\}$ such that 
\begin{equation*}
    \gamma_{j} \geq n+2-k \quad \textrm{and} \quad \gamma_{j+1}
    \leq -(n+2-k).
\end{equation*}
On the other hand, from condition \ref{it:ve-3} we obtain:
\begin{equation*}
    \gamma_{j}-\gamma_{j+1} = n+2-j+\lambda_{j} - (n+2-(j+1)+\lambda_{j+1}) \leq 1+\lambda_{j}-\lambda_{j+1}\leq 2(n+2-k).
\end{equation*}
Finally, we obtain:
\begin{equation*}
    \gamma_{j} = n+2-k \quad \textrm{and} \quad \gamma_{j+1}
    = -(n+2-k),
\end{equation*}
so we have $\abs{\gamma_{j}}=\abs{\gamma_{j+1}}$, contradicting the assumption that the absolute values of the entries of $\gamma$ are distinct.
\end{proof}

\subsection{Odd isotropic Grassmannian}\label{subsec:odd-grass}

Let $V$ be a ${(2n+1)}$-dimensional vector space endowed with $\psi \in \wedge^2 V^*$, a skew-symmetric form of maximal possible rank $2n$. Fix a $(2n+2)$-dimensional symplectic vector space $(\widetilde{V}, \widetilde{\psi})$ such that $V\subset \widetilde{V}$ is a hyperplane and $\widetilde{\psi}\vert_V=\psi$. 

The odd isotropic Grassmannian $X=\IGr(k, V)$ is the variety parametrising isotropic $k$-subspaces of $V$, that is, 
\begin{equation*}
    \IGr(k,V)=\Gr(k,V)\cap \IGr(k,\widetilde{V})
\end{equation*} 
where the intersection is considered in $\Gr(k,\widetilde{V})$.

We denote the stabilizer of $\psi$ in $\GL(V)$ by $\Sp(V)=\Sp_{2n+1}$ (\cf \cite[\textsection 3.3]{mihai2007odd}, \cite{proctor1988odd}) in analogy with the classical symplectic group. We refer to $\Sp_{2n+1}$ as the \textit{odd symplectic group}, it is connected and nonreductive.
The natural action of $\Sp_{2n+1}$  on $\IGr(k,V)$ is quasi-homogeneous.

\subsubsection{Realization}\label{subsec:odd-notn}

We can define $X$ as a zero locus of a global section of a vector bundle on $\IGr(k,\widetilde{V})$. We denote by $\widetilde{\cU}$ the tautological bundle of  $\IGr(k,\widetilde{V})$. Consider any nonzero section:
\begin{equation*}
    \widetilde{v} \in \widetilde{V}^*= \H^0(\IGr(k,\widetilde{V}), \, \widetilde{\cU}^*),
\end{equation*}
i.e. a nonzero linear form on $\widetilde{V}$, with $V = \Ker \widetilde{v}$. Then $\IGr(k,V)$ is the zero locus of $\widetilde{v}$. As $\IGr(k,V)$ is a smooth variety of expected codimension by \cite[Proposition~4.1]{mihai2007odd}, $\widetilde{v}$ is a regular section. We denote the embedding as $j:\IGr(k,V)\rightarrow \IGr(k,\widetilde{V})$. 

We have the following Koszul resolution in $\IGr(k,\widetilde{V})$:
\begin{equation}\label{eq:koszul}
0\rightarrow \wedge^k \widetilde{\cU} \rightarrow \wedge^{k-1} \widetilde{\cU} \rightarrow \dots \rightarrow \widetilde{\cU} \rightarrow \cO_{\IGr(k,\widetilde{V})}\rightarrow j_* \cO_{\IGr(k,V)} \rightarrow 0. 
\end{equation}
The following proposition provides an odd-dimensional analogue to \cref{lem:even-inv}.
\begin{proposition}\label{prop:invariants-easy}
The dimension $d$ and the index $w$ of $\IGr(k,2n+1)$ are respectively:
\begin{equation*}
    d=\frac{k(4n-3k+3)}{2}, \quad w = 2n + 2 - k.
\end{equation*}
Moreover, the Grothendieck group is a free abelian group and its rank $r$ is:
\begin{equation*}
    r =\binom{n}{k-1}\cdot\frac{2^{k-1}(2n+2-k)}{k}.
\end{equation*}
\end{proposition}
\begin{proof}
The first two formulas are well-known and can be obtained from the the embedding of $\IGr(k,V)$ in $\IGr(k, \widetilde{V})$, \cf \cite[Proposition~4.1]{mihai2007odd}. 
The formula for the rank of the  Grothendieck group is an immediate consequence of \cref{lem:even-inv} and \cref{prop:invariants}. We give a detailed proof of \cref{prop:invariants} in \cref{sec:appendix}.
\end{proof}

\begin{remark}\label{rmk:size}
A full exceptional collection induces a basis of the Grothendieck group. If a variety $X$ of index $w$ admits a full rectangular Lefschetz collection with a basis of $p$ elements, this induces a basis of $K_0(X)$ with $r = w p$ elements. In particular, the index divides the rank of the Grothendieck group.
By \cref{prop:invariants-easy}, in the case of the odd isotropic Grassmannians $\IGr(3,2n+1)$ it is possible to have a full rectangular Lefschetz collection only if
\begin{equation*}
    2n-1 \quad \mathrm{divides} \quad \binom{n}{2}\cdot\frac{4(2n-1)}{3} \quad \Leftrightarrow \quad 3 \quad \mathrm{divides} \quad 4\cdot\binom{n}{2}  = 2n(n-1),
\end{equation*}
this is equivalent to $n \equiv 0 \; \textrm{or} \; 1 \;\textrm{mod} \; 3$. In this case, $p = \frac{2}{3}n(n-1)$.
\end{remark}

For $\IGr(3,9)$ we have $d=15$, $w=7$, $r=56$; therefore  $p=8$. 

\subsubsection{Vanishing}
We give here an acyclicity criterion similar to \cref{prop:vanish-even},  which will cover most bundles on the odd Grassmannian $X=\IGr(k,2n+1)$ studied in this work.

\begin{corollary}\label{prop:vanish-crit}
Let $\lambda=(\lambda_1,\dots,\lambda_k)\in P_k^+$ be a dominant weight of the group $\GL_k$ such that
\begin{enumerate}
    \item $\lambda_k < 0$,
    \item $\lambda_1 \geq - 2n + k-1$,
    \item $\lambda_i-\lambda_{i+1} \leq 2(n+1-k)$ for $i=1,\dots,k-1$.
\end{enumerate}
Then the bundle $\cU^\lambda$ on $X=\IGr(k, 2n+1)$ is acyclic.
\end{corollary}

\begin{proof}
Fix an embedding $j:X\rightarrow \IGr(k,2n+2)=\widetilde{X}$. Recall the notation fixed in \cref{subsec:odd-notn}. By the projection formula, we have an isomorphism:
\begin{equation*}
\H^\bullet(X,\cU^\lambda)= \H^\bullet(\widetilde{X}, j_* \cU^\lambda)= \H^\bullet(\widetilde{X},  \widetilde{\cU}^\lambda\otimes j_*\cO_X).
\end{equation*}
Replacing $j_*\cO_X$ by its Koszul resolution \eqref{eq:koszul}, we obtain the following spectral sequence:
\begin{equation}\label{eq:koszul-vanish}
    E^{-p,q}_1=\H^q(\widetilde{X}, \widetilde{\cU}^\lambda \otimes \wedge^p  \widetilde{\cU})\Rightarrow \H^{q-p}(X, \cU^\lambda)
\end{equation}
for $q\geq 0$ and $p=0,\dots,k$. Hence, it is enough to compute the cohomology of the direct summands \begin{equation*} 
\widetilde{\cU}^\gamma \inplus \widetilde{\cU}^\lambda \otimes \wedge^p  \widetilde{\cU} = \widetilde{\cU}^\lambda \otimes \wedge^k  \widetilde{\cU}\otimes (\wedge^{k-p}  \widetilde{\cU})^* = \widetilde{\cU}^{\lambda-(1,\dots,1)} \otimes (\wedge^{k-p}  \widetilde{\cU})^*. 
\end{equation*} By \cref{prop:pieri} and \cref{cor:LR-det}, it is enough to verify that the bundles $\widetilde{\cU}^\gamma$ are acyclic for any  
\begin{equation*} 
\gamma \in \bigcup\limits_{p\in\{0,\dots,k\}}\VS_{\lambda-(1,\dots,1)}^p. 
\end{equation*} 
Let us verify the conditions of
\cref{prop:vanish-even} for $\widetilde{\cU}^\gamma$. 
First of all, we have 
\begin{equation}\label{eq:VS}
    \lambda_i - 1\leq \gamma_i \leq  \lambda_i \quad \text{for} \quad i=1,\dots,k.
\end{equation}
As $\gamma_k\leq \lambda_k<0,$
by \eqref{eq:VS} and the hypothesis, the first condition is satisfied. Using the second hypothesis, we obtain:
\begin{equation*}
     (- 2n + k -1) -1 \leq \lambda_1 - 1\leq \gamma_1,
\end{equation*}
so the second condition of \cref{prop:vanish-even} is satisfied. Finally,
\begin{equation*}
     \gamma_{i}-\gamma_{i+1}\leq \lambda_i-\lambda_{i+1}+1\leq 2(n+1-k)+1,
\end{equation*}
so that the last condition is verified and $\widetilde{\cU}^\gamma$ is acyclic by \cref{prop:vanish-even}. Therefore, the spectral sequence \eqref{eq:koszul-vanish} vanishes at the first page, proving that $\cU^\lambda$ is acyclic.
\end{proof}

\section{Exceptionality}\label{sec:exc}

Let us fix $X=\IGr(3,9)$, the odd isotropic Grassmannian of $3$-dimensional subspaces in $V$, where $V$ is a $9$-dimensional vector space endowed with a skew-symmetric form of maximal rank $8$, $\psi$.
By \cref{prop:invariants-easy}, $X$ has index $w=7$. Recall that in our notation we have:
\begin{equation*}
 \cU^{m,0,0}=S^m \cU^*,\quad \cU^{0,0,-m}=S^m \cU, \quad  \cU^{l,l,l}=\cO(l).
\end{equation*}

We will often need to consider the embedding $j:\IGr(3,9) \rightarrow \IGr(3,10)$. We will denote the tautological bundle of $\IGr(3,10)$ as $\wcU$.
We restate here \cref{prop:vanish-even} and \cref{prop:vanish-crit} for $\IGr(3,10)$ and $\IGr(3,9)$, they will be used as key tools to show semiorthogonality in \cref{sec:twocoll}.
\begin{corollary}\label{vanish-IG310}
Let $(\lambda_1,\lambda_2,\lambda_3)\in P_3^+$ be a dominant weight of the group $\GL_3$ such that
\begin{enumerate}
    \item $\lambda_3 < 0$,
    \item $\lambda_1 \geq - 7$,
    \item $\lambda_1-\lambda_2 \leq 5$ and $\lambda_2-\lambda_3 \leq 5$.
\end{enumerate}
Then, $\H^\bullet(\IGr(3, 10),\widetilde{\cU}^{\lambda})=0$. If $\lambda_1\geq -1$ instead, $\H^\bullet(\IGr(3, 10),\widetilde{\cU}^{\lambda}(-l))=0$ for $l=0,\dots,6$.
\end{corollary}
\begin{corollary}\label{vanish-IG39}
Let $(\lambda_1,\lambda_2,\lambda_3)\in P_3^+$ be a dominant weight of the group $\GL_3$ such that
\begin{enumerate}
    \item $\lambda_3 < 0$,
    \item $\lambda_1 \geq - 6$,
    \item $\lambda_1-\lambda_2 \leq 4$ and $\lambda_2-\lambda_3 \leq 4$.
\end{enumerate}
Then, $\H^\bullet(\IGr(3, 9),\cU^{\lambda})=0$. If $\lambda_1\geq 0$ instead, $\H^\bullet(\IGr(3, 9),\cU^{\lambda}(-l))=0$ for $l=0,\dots,6$.
\end{corollary}

\subsection{Two collections}\label{sec:twocoll}

Recall that in the \cref{sec:intro} we defined two collections in $\DX$ of length $7$:
\begin{equation*}
\begin{aligned}
    \bB_1 &= \{ \ \cU^{0,0,-2}, &\cU^{0,0,-1},\quad &\cU^{1,0,-1}, &\cU^{2,0,-1},\quad &\cU^{0,0,0},  &\cU^{1,0,0},\quad &\cU^{2,0,0} &\}, \\
    \bB_2 &= \{&\cU^{0,0,-1},\quad &\cU^{1,0,-1}, &\cU^{2,0,-1},\quad &\cU^{0,0,0},  &\cU^{1,0,0},\quad
    &\cU^{2,0,0}, &\cU^{3,0,0} \ \}.
\end{aligned}
\end{equation*}
We remark that all the weights $\lambda=(\lambda_1,\lambda_2,\lambda_3)$ such that $\cU^\lambda\in \bB_1\cup \bB_2$ satisfy $\lambda_2=0$. Notice that 
\begin{equation*}
\bB_1=\{\,\cU^{0,0,-2}\,\}\cup (\,\bB_1\cap \bB_2\,),\quad \quad  
\bB_2= (\,\bB_1\cap \bB_2\,)\cup \{\,\cU^{3,0,0}\,\}.
\end{equation*} Moreover, $\bB_1$ and $\bB_2$ are ordered lexicographically with entries read from right to left, i.e.:
\begin{equation}\label{eq:except-order}
    (\beta_1,\beta_2,\beta_3) < (\alpha_1,\alpha_2,\alpha_3)  \Leftrightarrow \begin{cases*}\; \beta_3 < \alpha_3 \quad \text{or} \\
    \; \beta_3 = \alpha_3 \quad  \text{and} \quad \beta_2 < \alpha_2 \quad \text{or} \\
    \; \beta_3 = \alpha_3 \quad \text{and} \quad \beta_2 = \alpha_2 \quad \text{and} \quad \beta_1 < \alpha_1
    \end{cases*}
\end{equation}
which refines the partial order $\subseteq$ on $P^+_3$ presented in \eqref{eq:order}. Recall that an additional object $\cH\in \DX$ was introduced as the totalization of the bicomplex \eqref{H-intro} and that the collection 
\begin{equation*}
\bB=\{\,\cH\,\} \cup \bB_1
\end{equation*}
was defined in \eqref{def:B-intro}.

The first part of the section shows that $\bB_1$ and $\bB_2$ are bases of two (non full) rectangular Lefschetz exceptional collections (\cf \cref{cor:bases}). In addition, we prove the following intermediate statement.\begin{proposition}\label{prop:Semi-O}
Let $\Bsc\subseteq \DX$ be the subcategory generated by $\bB$. Then the collection of subcategories
$\Bsc,\, \Bsc(1),\, \Bsc(2),\, \Bsc(3),\, \Bsc(4),\, \Bsc(5),\, \Bsc(6)$
is semiorthogonal.
\end{proposition}
Notice that this proposition does not claim that $\Bsc$ is generated by an exceptional collection or that $\Bsc$ is admissible. These properties are indeed true, but they will only be proved in \cref{thm:main}, after a more detailed study of $\cH$. 

To prove \cref{prop:Semi-O}, the starting point is to observe that $\cH$ lies in the category $\langle \, \bB_2 \, \rangle \subset \Bsc$, so that $\bB_1 \cup \bB_2$ is an alternative set of generators of $\Bsc$. Therefore, we need to compute $\Ext^\bullet$-groups between various twists of the generators of $\bB_1$ and $\bB_2$. We will need some ad hoc computations of cohomology, which we group here at the beginning of the section. 

We start with some computations on $\IGr(3,10)$. Recall the notation introduced in \cref{subsec:odd-notn}.

\begin{lemma}\label{lem:homog-600}\label{lem:homog-0-1-7}
We have the following isomorphisms:
\begin{align*}
    \H^\bullet(\IGr(3,10),\, \widetilde{\cU}^{0,0,-6}(-l))&=\begin{cases*}
      \mathbb{C}[-5] & if $l=0$, \\
      0        & if $l=1,\dots,7$,
    \end{cases*}\\
    \H^\bullet(\IGr(3,10),\, \widetilde{\cU}^{0,-1,-7}(-l))&= \begin{cases*}
      \mathbb{C}[-6] & if $l=0$, \\
      0        & if $l=1,\dots,7$.
    \end{cases*}
\end{align*}
\end{lemma}
\begin{proof}
To prove both results, we apply \cref{thm:BWB}. We focus on the first isomorphism. To do so, take $\lambda=(-l,-l,-l-6),$ and consider $\gamma=(-l,-l,-(l+6),0,0)$ its trivial extension, so that:
\begin{equation*}
    \gamma+\rho_{\Sp_{10}}=(5-l,4-l,-l-3,2,1).
\end{equation*}
If $l=0$ we have $\gamma+\rho_{\Sp_{10}}=(5,4,-3,2,1)$, so all the entries are distinct in absolute value and nonzero. Let $\sigma$ be the element of the Weil group such that $\sigma(5,4,-3,2,1)=(5,4,3,2,1)\in T_5^+$, hence:
\begin{equation*}
    \gamma'=\sigma(5,4,-3,2,1)-\rho_{\Sp_{10}}=(0,0,0,0,0).
\end{equation*}
We observe that $\ell(\sigma)=5$, obtaining the result for $l=0$. On the other hand, for $l=1,\dots,7$, there is always a repetition of absolute values. The vanishing for these $l$ follows. 

In the second case, $\gamma=(-l,-1-l,-7-l,0,0)$, obtaining
\begin{equation*}
    \gamma+\rho_{\Sp_{10}}=(5-l,3-l,-4-l,2,1),
\end{equation*}
and the same argument applies. 
\end{proof}

We continue with simple cohomology computations on $\IGr(3,9)$.
\begin{lemma}\label{lem:dishomog-00-5}\label{lem:dishomog-0-1-6}
We have the following isomorphism:
\begin{align*}
     \H^{\bullet}(\IGr(3,9),\cU^{0,0,-5}(-l))&= \begin{cases*}
      \mathbb{C}[-4] & if $l=0$, \\
      0        & if $l=1,\dots,6$,
      \end{cases*}\\
      \H^\bullet(\IGr(3,9),\cU^{0,-1,-6}(-l))&= \begin{cases*}
      \mathbb{C}[-5] & if $l=0$, \\
      0        & if $l=1,\dots,6$.
      \end{cases*} 
\end{align*}
\end{lemma}

\begin{proof}
Let us fix an embedding $j:\IGr(3,9) \rightarrow \IGr(3,10)$ and consider the spectral sequence  \eqref{eq:koszul-vanish}, which we rewrite here for convenience:
\begin{equation*}
    E^{-p,q}_1=\H^q(\IGr(3,10), \widetilde{\cU}^\lambda \otimes \wedge^p  \widetilde{\cU}(-l))\Rightarrow \H^{q-p}(\IGr(3,9), \cU^\lambda(-l)),
\end{equation*} 
where $\widetilde{\cU}$ is the tautological bundle on $\IGr(3,10)$. We prove that the spectral sequence for $\lambda=(0,0,-5)$ and $\lambda=(0,-1,-6)$ either vanishes or, if $l=0$, it has only one nonzero term. By \cref{prop:pieri} and \cref{cor:LR-det} (compare with the proof of \cref{prop:vanish-crit}), the terms in the spectral sequence are the cohomologies on $\IGr(3,10)$ of the bundles $\widetilde{\cU}^\gamma(-l)$ for 
\begin{equation*} 
\gamma \in \bigcup\limits_{p\in\{0,\dots,k\}}\VS_{\lambda-(1,\dots,1)}^p = \{(\gamma_1,\gamma_2,\gamma_3)\mid \gamma_1\geq\gamma_2\geq \gamma_3, \quad \lambda_i\geq \gamma_i\geq \lambda_i-1\}. 
\end{equation*} 
In particular, for both $\lambda$, we have $\gamma_1 \geq \lambda_1-1 \geq -1$ and $-5\geq \lambda_3 \geq \gamma_3$. 
Moreover, we have:
\begin{equation*}
    \gamma_1-\gamma_2\leq (\lambda_1-\lambda_2) + 1  \leq 2, \quad \quad \gamma_2-\gamma_3 \leq (\lambda_2-\lambda_3) + 1 \leq 6.
\end{equation*}
If $ \gamma_2-\gamma_3\leq 5$,   \; $\widetilde{\cU}^\gamma(-l)$ is acyclic for $l=0,\dots,6$ by \cref{vanish-IG310}. 
It is immediate to determine the bundles appearing in both spectral sequences which satisfy $\gamma_2-\gamma_3=6$:
\begin{itemize}
    \item $\widetilde{\cU}^{0,0,-6}(-l)\inplus \widetilde{\cU}^{0,0,-6}\otimes \widetilde{\cU} (-l)$ if $\lambda=(0,0,-5)$;
    \item $\widetilde{\cU}^{0,0,-6}(-l-1)\inplus \widetilde{\cU}^{0,-1,-6}\otimes \wedge^2 \widetilde{\cU} (-l)$ and $\widetilde{\cU}^{0,-1,-7}(-l)\inplus \widetilde{\cU}^{0,-1,-6}\otimes \widetilde{\cU} (-l)$ in the case ${\lambda=(0,-1,-6)}$.
\end{itemize}
By \cref{lem:homog-600}, the spectral sequence vanishes identically for both $\lambda$ if $l=1,\dots,6$. If $l=0$, there is one nonzero entry in both spectral sequences by \cref{lem:homog-600}, proving the result.
\end{proof}

We need some extra computations of cohomology for \cref{sec:H-props}.
\begin{lemma}\label{lem:tensor1}\label{lem:SC^4-1-1}
We have the following isomorphisms:
\begin{align*}
\H^\bullet(\IGr(3,9),\,\cU^{1,0,-2}\otimes\cU^{0,0,-3}(-1))&=\mathbb{C}[-5],\\
\H^\bullet(\IGr(3,9),\,\cU^{1,0,-3}\otimes \cU^\lambda(-2))&=0,
\end{align*}
with $\lambda\in \{(1,0,-1),\,(1,0,0),\,(2,0,0)\} \subset P_3^+$.
\end{lemma}
\begin{proof}
We prove the first equality. Applying Pieri's formula (\cref{prop:pieri}), we obtain:
\begin{align*}
    \cU^{1,0,-2} \otimes \cU^{0,0,-3}(-1) &= \cU^{0,-1,-6}\oplus \cU^{0,-2,-5} \oplus  \cU^{-1,-1,-5} \oplus \\ &\oplus \cU^{0,-3,-4} \oplus \cU^{-1,-2,-4} \oplus \cU^{-1,-3,-3}.
\end{align*}
By \cref{vanish-IG39}, all the bundles in the decomposition are acyclic except possibly $\cU^{0,-1,-6}$. Applying \cref{lem:dishomog-0-1-6}, we deduce the first equality. 

We now prove the last three vanishings. Let $\gamma=(\gamma_1,\gamma_2,\gamma_3)\in P_3^+$ such that:
\begin{equation*}
    \cU^\gamma \inplus \cU^{1,0,-3}\otimes \cU^\lambda(-2),
\end{equation*}
for $\lambda\in\{(1,0,-1),\,(1,0,0),\,(2,0,0)\}$.
We can deduce the following facts from \cref{lem:sum}, fixing $\alpha=(5,2,1)=(3,0,-1)+(2,2,2)$ and $\beta=\lambda$ as above:
\begin{equation*}
    \beta_1-1 \geq \gamma_1 \geq \beta_3-1, \quad
    \beta_2-1 \geq \gamma_2 \geq \beta_3-2, \quad
    \beta_3-1 \geq \gamma_3 \geq \beta_3-5, \quad \abs{\gamma}=\abs{\beta}-8.
\end{equation*}
It is immediate to see that any possible $\cU^\gamma$ is acyclic by \cref{vanish-IG39}, unless $\gamma=(-1,-1,-6)$ and $\beta=(1,0,-1)$. The bundle $\cU^{-1,-1,-6}$ is acyclic by \cref{lem:dishomog-00-5}, proving the claim.
\end{proof}

The next lemma shows that $\bB_1 \cup \bB_2$ is not an exceptional collection; that is the only reason stopping $\bB_1\cup \bB_2$ from extending to a Lefschetz collection.

\begin{lemma}\label{lem:extremal computation}
There is the following isomorphism:
\begin{equation*}
    \Ext^\bullet(\cU^{3,0,0}(l),\cU^{0,0,-2})=\begin{cases*}
      \mathbb{C}[-4] & if $l=0$, \\
      0        & if $l=1,\dots,6$.
    \end{cases*}
\end{equation*}
\end{lemma}
\begin{proof}
We have $(\cU^{3,0,0})^*\cong \cU^{0,0,-3}$ and by Pieri's formula:
\begin{equation*}
\cU^{0,0,-3}\otimes\cU^{0,0,-2} = \cU^{0,0,-5} \oplus \cU^{0,-1,-4} \oplus \cU^{0,-2,-3}.
\end{equation*}
Using \cref{vanish-IG39}, we verify that for all bundles $\cU^{\gamma}\inplus \cU^{0,0,-3}\otimes\cU^{0,0,-2}$, the bundle $\cU^{\gamma}(-l)$ is acyclic for $l=0,\dots,6$, except for $\gamma=(0,0,-5)$. For $\gamma=(0,0,-5)$, we apply \cref{lem:dishomog-00-5} and we deduce the claim.
\end{proof}

In light of the vanishing criterion given by \cref{vanish-IG39}, we highlight the main property of the collections $\bB_1$ and $\bB_2$.

\begin{lemma}\label{prop:max-jumps}
Let $\alpha,\beta \in P^+_3$ such that $\cU^{\alpha}, \cU^{\beta} \in \bB_j$ for the same $j=1,2$. Suppose that
\begin{equation*}
    \cU^\gamma \inplus \cU^{-\alpha} \otimes \cU^{\beta}. 
\end{equation*}
Then, $\gamma$ satisfies:
\begin{equation*}
    \gamma_1-\gamma_2 \leq 4 \quad \text{and} \quad \gamma_2-\gamma_3 \leq 4.
\end{equation*}
\end{lemma}
\begin{proof}
Let $\lambda$ be such that $\cU^\lambda \in \bB_1\cup \bB_2$, then it satisfies 
\begin{equation*}
    \lambda_1-\lambda_3\leq 3 \quad \textrm{and} \quad \lambda_2-\lambda_3\leq 2,
\end{equation*}
where the first equality is attained for $\lambda=(2,0,-1)\in \bB_1\cap \bB_2$ and $\lambda=(3,0,0)\in \bB_2 \setminus \bB_1$, while the second only for $\lambda=(0,0,-2)\in \bB_1 \setminus \bB_2$. 
By \eqref{eq:jumps}, we have:
\begin{equation*} 
    \gamma_1-\gamma_2 \leq 4 \quad \textrm{and} \quad \gamma_2-\gamma_3 \leq 4,
\end{equation*}
except possibly for $\alpha = (0,0,-2)$ and $\beta = (2,0,-1)$ or $(3,0,0)$ or vice versa. This finishes the proof for $\cU^{\alpha}, \cU^{\beta} \in\bB_2$.
Finally, if $\alpha=(0,0,-2)$ and $\beta=(2,0,-1)$, we have by Pieri's formula:
\begin{equation*}
    \cU^{2,0,0}\otimes \cU^{2,0,-1} = \cU^{4,0,-1}\oplus \cU^{3,1,-1}\oplus \cU^{3,0,0}\oplus \cU^{2,2,-1}\oplus\cU^{2,1,0},
\end{equation*}
so that we can verify the claim.
\end{proof}
Now we finished preparations and we can  prove the Lefschetz property for $\bB_1$ and $\bB_2$. 
\begin{proposition}\label{prop:long}
Let $\alpha,\beta \in P^+_3$, with $\beta \leq \alpha$ with respect to \eqref{eq:except-order}, such that $\cU^{\alpha}, \cU^{\beta} \in \bB_j$ for the same $j=1,2$. Then 
\begin{equation*}
\Ext^\bullet(\cU^\alpha(l),\cU^\beta)=0 \quad \text{for} \quad l=0,\dots 6,
\end{equation*}
unless $l=0$ and $\alpha=\beta$. In that case, we have:
\begin{equation*}
\Ext^\bullet(\cU^\alpha,\cU^\alpha)= 
    \mathbb{C}. 
\end{equation*}
\end{proposition}

\begin{proof}
With the notation presented above, consider the decomposition:
\begin{equation*}
\cU^{\gamma}\inplus \cU^{-\alpha}\otimes \cU^{\beta}.
\end{equation*}
We claim that either:
\begin{itemize}
\item $\gamma_1 \geq 0$ and $\gamma_3<0$; or\item $\gamma=(0,0,0)$ and $\alpha = \beta$, and it appears with multiplicity one.
\end{itemize}
Applying \cref{lem:sum} and recalling $\beta \leq \alpha$, we obtain:
\begin{equation}\label{eq:intermediate}
     \gamma_1 \geq \beta_3-\alpha_3, \quad
    0\geq \beta_3-\alpha_3 \geq \gamma_3, \quad \gamma_1+\gamma_2+\gamma_3=\abs{\gamma}=\abs{\beta}-\abs{\alpha}.
\end{equation}
We first prove that every $\gamma$ satisfies $\gamma_1\geq 0$. If $\alpha_3=\beta_3$, we obtain $\gamma_1 \geq 0$. Otherwise, at least one between $\alpha_3,\beta_3 \neq -1$, so we obtain that either $\cU^{-\alpha}=S^{\alpha_1} \cU$ if $\alpha_3=0$ or $\cU^{\beta}=S^{-\beta_3} \cU$ if $\beta_3=-2$, while the other factor is some $\cU^\lambda \in \bB_1\cup \bB_2$, with $\lambda_2=0$. Applying Pieri's formula, we conclude that $\gamma_1\geq \lambda_2=0$.  This proves  that all $\gamma$ 
satisfy $\gamma_1 \geq 0$.

We now focus on $\gamma_3$. If $\gamma_3=0$, then $\beta_3=\alpha_3$ by \eqref{eq:intermediate} and $\beta_1\leq \alpha_1$ by \eqref{eq:except-order}, obtaining  $\abs{\gamma}\leq0$. This shows $\gamma=(0,0,0)$. By \cref{lem:Schur}, we obtain $\alpha=\beta$ and $\gamma$ has multiplicity one. Alternatively, by \eqref{eq:intermediate}, we have $\gamma_3<0$, hence $\gamma$ falls in the first case of the initial claim. 

By \cref{prop:max-jumps} we already know $\gamma_1-\gamma_2\leq 4$ and $\gamma_2-\gamma_3\leq 4$, applying \cref{vanish-IG39}, this is enough to conclude. 
\end{proof}

The following corollary is an immediate consequence of \cref{prop:long} and \cref{lem:criterion-lef}. 

\begin{corollary}\label{cor:bases}
The collections 
\begin{equation*}
\begin{aligned}
    \bB_1 &= \{ \ \cU^{0,0,-2}, &\cU^{0,0,-1},\quad &\cU^{1,0,-1}, &\cU^{2,0,-1},\quad &\cU^{0,0,0},  &\cU^{1,0,0},\quad &\cU^{2,0,0} &\}, \\
    \bB_2 &= \{&\cU^{0,0,-1},\quad &\cU^{1,0,-1}, &\cU^{2,0,-1},\quad &\cU^{0,0,0},  &\cU^{1,0,0},\quad
    &\cU^{2,0,0}, &\cU^{3,0,0} \ \}.
\end{aligned}
\end{equation*}
are bases of rectangular Lefschetz exceptional collections of length $7$.
\end{corollary}

We conclude the section proving \cref{prop:Semi-O}.
\begin{proof}[Proof of \cref{prop:Semi-O}.]
As $\cH$ is the totalization of the bicomplex \eqref{H-intro}, $\cH \in \langle \, \bB_2 \, \rangle$ because it admits a resolution where all terms belong to $\bB_2$. On the other hand, the same resolution of $\cH$ shows that $\cU^{3,0,0}\in \Bsc$, the category generated by $\bB$ (which is possibly not an exceptional sequence, \cf \cref{thm:main}). Hence, the subcategory $\Bsc$ can be generated by  $\bB_1\cup \bB_2$ (which is not an exceptional sequence). 

By \cref{lem:criterion-lef}, to prove semiorthogonality it is enough to verify 
\begin{equation}\label{eq:test} \Ext^\bullet(\cU^\alpha(l), \cU^\beta)=0\quad \textrm{for} \quad \beta \leq \alpha \quad \textrm{and} \quad 1\leq l\leq 6,
\end{equation}
with $\cU^\alpha, \cU^\beta \in \bB_1\cup \bB_2$. By \cref{cor:bases}, the equation \eqref{eq:test} holds if both $\alpha,\beta \in \bB_1$ or both $\alpha, \beta \in \bB_2$. On the other hand, if $\alpha=(3,0,0)$ and $\beta=(0,0,-2)$, 
we have 
\begin{equation*}
    \Ext^\bullet(\cU^{3,0,0}(l), \cU^{0,0,-2})=0\quad \textrm{for} \quad 1\leq l\leq 6
\end{equation*}
by \cref{lem:extremal computation}, proving the initial statement.
\end{proof}

\subsection{Properties of $\cH$}\label{sec:H-props}

Recall the definition of the object $\cH$ as convolution of the bicomplex \eqref{H-intro}. The purpose of this section is to show that $\cH$ is an exceptional object (\cf \cref{excep}) completing $\bB_1$ to a Lefschetz basis (\cf \cref{thm:main}), hence we need to show that $\cH$ is exceptional and 
$\cH \in \bB_1^\perp$.

We first outline the properties of the lines of the bicomplex. Let us consider the (acyclic!) staircase complex on $\IGr(3,9)$ (\cf \cref{thm:staircase}) associated to $\cU^{3,0,0}$:
\begin{multline}\label{eq:E}
    0 \rightarrow \cU^{0,0,-3}(-1) \rightarrow \wedge^8 V^* \otimes \cU^{0,0,-2}(-1) \rightarrow \wedge^7 V^* \otimes \cU^{0,0,-1}(-1) \rightarrow \wedge^6 V^* \otimes \cO(-1) \rightarrow  \\ \rightarrow \wedge^3V^* \otimes \cO \rightarrow \wedge^2V^* \otimes \cU^{1,0,0} \rightarrow V^* \otimes \cU^{2,0,0} \rightarrow  \cU^{3,0,0} \rightarrow 0.
\end{multline}
The complex \eqref{eq:E} is self-dual, as it is the unique complex with nonzero $\GL$-equivariant differentials and these terms (see \cref{thm:staircase}). We define the object $\cE$ as the stupid truncation of the complex between the lines. As the complex is exact, this induces two resolutions of $\cE$ by $\Sp_9$-equivariant bundles (\cf \cref{subsec:odd-grass}). More explicitly, there are two exact sequences:
\begin{multline}\label{eq:left E}
     0 \rightarrow \cU^{0,0,-3}(-1) \rightarrow \wedge^8 V^* \otimes \cU^{0,0,-2}(-1) \rightarrow \\ \rightarrow \wedge^7 V^* \otimes \cU^{0,0,-1}(-1) \rightarrow \wedge^6 V^* \otimes \cO(-1) \rightarrow \cE \rightarrow 0
\end{multline}
and 
\begin{equation}\label{eq:right E}
0\rightarrow \cE \rightarrow \wedge^3V^* \otimes \cO \rightarrow \wedge^2V^* \otimes \cU^{1,0,0} \rightarrow V^* \otimes \cU^{2,0,0} \rightarrow  \cU^{3,0,0} \rightarrow 0.
\end{equation}
As we can deduce from the self-duality of \eqref{eq:E},
\begin{equation}\label{property E}
    \cE\cong \cE^*(-1).
\end{equation}

We now consider the (acyclic!) staircase complex on $\IGr(3,9)$ induced by $\cU^{2,0,-1}$:
\begin{multline}\label{eq:staircase 20-1}
    0\rightarrow \cU^{1,0,-3}(-2) \rightarrow \wedge^8 V^*\otimes \cU^{1,0,-2}(-2) \rightarrow \\ \rightarrow  \wedge^7 V^*\otimes \cU^{1,0,-1}(-2)\rightarrow \wedge^6 V^*\otimes \cU^{1,0,0}(-2) \rightarrow 
    \wedge^4 V^*\otimes \cO(-1) \\ \rightarrow \wedge^2V^* \otimes \cU^{0,0,-1} \rightarrow V^* \otimes \cU^{1,0,-1} \rightarrow  \cU^{2,0,-1} \rightarrow 0.
\end{multline}
The object $\cF$ is the stupid truncation of the complex above between the middle and the lowest line. We obtain two resolutions of $\cF$. More explicitly, there are two exact sequences:
\begin{multline}\label{eq:left F}
    0\rightarrow \cU^{1,0,-3}(-2) \rightarrow \wedge^8 V^*\otimes \cU^{1,0,-2}(-2) \rightarrow \wedge^7 V^*\otimes \cU^{1,0,-1}(-2)\rightarrow\\ \rightarrow  \wedge^6 V^*\otimes \cU^{1,0,0}(-2) \rightarrow 
    \wedge^4 V^*\otimes \cO(-1)\rightarrow \cF\rightarrow 0
\end{multline}
and 
\begin{equation}\label{eq:right F}
0 \rightarrow \cF \rightarrow  \wedge^2V^* \otimes \cU^{0,0,-1} \rightarrow V^* \otimes \cU^{1,0,-1} \rightarrow  \cU^{2,0,-1} \rightarrow 0.    
\end{equation}
We observe that \begin{equation}\label{eq:E-F-B}
\cE, \cF\in \Bsc 
\end{equation} 
because all the terms in their right resolutions \eqref{eq:right E} and \eqref{eq:right F} belong to $\Bsc$. Notice that $\cE$ and $\cF$ are vector bundles by \cite[Theorem~19.2]{eisenbud2013commutative}. 

We first compute the $\Ext^\bullet$-groups between $\cE$ and $\cF$ and show that the sequences \eqref{eq:right E} and \eqref{eq:right F} induce mutation triangles. To do so, we introduce the following subsets of $\bB_1$:
\begin{equation}\label{eq:def-S}
    \begin{aligned}
    \bS_1 &= \{ \cU^{0,0,0}, \cU^{1,0,0}, \cU^{2,0,0}  \}, \\
    \bS_2 &= \{  \cU^{0,0,-2} , \cU^{0,0,-1}, \cU^{1,0,-1} \}, \\
    \bS &= \bB_1 \setminus \{ \cU^{2,0,-1} \} = \bS_1 \cup \bS_2.
\end{aligned}
\end{equation}
An observation that simplifies many computations in this section is that:
\begin{equation}\label{eq:S-property}
    \bS^* = \bS \subset \Bsc,
\end{equation}
where by $\bS^*$ we denote all the duals of the elements in $\bS$.

\begin{proposition}\label{prop:mutate}
The bundle $\cE$ is right orthogonal to $\cF$:
\begin{equation*}
    \Ext^\bullet(\cE,\cF)=0.
\end{equation*}
Moreover, $\cE$ and $\cF$ are right orthogonal to $\bS$, that is:
\begin{equation*}
    \Ext^\bullet(\bS, \cE)= \Ext^\bullet(\bS, \cF)=0.
\end{equation*}
Finally, the following isomorphisms hold:
\begin{equation*}
    \cE[3] = \mathbb{L}_{\bS_1\cap \bB_2} \;\cU^{3,0,0} = \mathbb{L}_{\bS} \;\cU^{3,0,0} , \quad \quad \cF[2] =\mathbb{L}_{\bS_2 \cap \bB_2} \;\cU^{2,0,-1} = \mathbb{L}_{\bS} \;\cU^{2,0,-1}.
\end{equation*}
\end{proposition}

\begin{proof}
As $\cE \cong \cE^*(-1)$ by \eqref{property E}, we have:
\begin{equation*}
    \Ext^\bullet(\bS, \cE) = \Ext^\bullet(\bS, \cE^*(-1)) = \Ext^\bullet(\cE(1), \bS^*) = \Ext^\bullet(\cE(1), \bS),
\end{equation*}
where the last equality holds by the symmetry of $\bS$, \cf \eqref{eq:S-property}.
As $\cE(1)\in \Bsc(1)$ by \eqref{eq:E-F-B} and $\bS\subset \Bsc$, the last $\Ext^\bullet$-group vanishes by \cref{prop:Semi-O}, proving that $\cE\in \bS^\perp$.

We now focus on $\cF$. We observe that:
\begin{equation*}
    \Ext^\bullet(\bS_1, \cF) = 0, \quad \quad \Ext^\bullet(\cU^{3,0,0}, \cF) = 0,
\end{equation*}
because $\bB_2$ is an exceptional collection by \cref{cor:bases} and
all the terms in the right resolution  \eqref{eq:right F} of $\cF$ belong to $\langle \; \cU^{0,0,-1},\; \cU^{1,0,-1},\; \cU^{2,0,-1}\; \rangle$, as a consequence, we obtain that  \begin{equation*}
    \cF \in \langle\; \cU^{0,0,0},\; \cU^{1,0,0},\; \cU^{2,0,0},\; \cU^{3,0,0}\;\rangle ^\perp = \bS_1^\perp \, \cap \,  (\,\cU^{3,0,0}\,)^\perp.
\end{equation*} Finally, $\cE\in \langle\; \cU^{0,0,0},\; \cU^{1,0,0},\; \cU^{2,0,0},\; \cU^{3,0,0}\;\rangle $ by \eqref{eq:right E}, hence we obtain $\cF\in \cE^\perp$.

Notice that we have:
\begin{equation}\label{eq:intermezzo}
    \Ext^\bullet(\bS_2, \cF) \cong \Ext^\bullet(\cF^*, \bS_2^*).
\end{equation}
By \eqref{eq:left F}, we have that $\cF^*$ belongs to the category generated by $\bB_2(1),\, \bB_2(2)$ and $\cU^{3,0,-1}(2)$,
while $\bS_2^* \subset \langle\, \bB_2\, \rangle$. As a consequence, if $\Ext^\bullet(\cU^{3,0,-1}(2),\bS_2^*)=0$, the right side of \eqref{eq:intermezzo} vanishes by \cref{prop:Semi-O}. But this follows from  \cref{lem:tensor1}, proving that $\cF \in \bS^\perp$.

Finally, by \eqref{eq:right E} and \eqref{eq:right F}, we have:
\begin{equation*}
    \Cone(\, \cU^{3,0,0}\rightarrow \cE[3]\,)\in \langle \, \bS_1\cap \bB_2 \, \rangle , \quad \quad \Cone(\, \cU^{2,0,-1} \rightarrow \cF[2] \,)\in \langle \, \bS_2\cap \bB_2 \, \rangle.
\end{equation*}
As we proved that $\cE, \cF\in \bS^\perp$, the sequences \eqref{eq:right E} and \eqref{eq:right F} induce mutation triangles.
\end{proof}

\begin{proposition}\label{prop:ext-isom}
We have the following isomorphisms of $\Ext^\bullet$-groups:
\begin{equation*}
   \Ext^\bullet(\cF,\cE)= \mathbb{C}.
\end{equation*}
\end{proposition}
\begin{proof}
By \eqref{eq:left E} and \eqref{eq:right F}, we have
\begin{align*}
    \Cone(\,\cU^{2,0,-1}[-2] \rightarrow \cF \,)\in &\langle\, \cU^{0,0,-1},\,\cU^{1,0,-1}\, \rangle  \subset \Bsc\cap \Bsc^*,\\
    \Cone(\,\cE\rightarrow \cU^{0,0,-3}(-1)[3]\,)\in &\langle\, \cU^{0,0,-2}(-1),\,\cU^{0,0,-1}(-1),\,\cO(-1)\, \rangle \subset \Bsc(-1).
\end{align*}
Applying \cref{prop:Semi-O} and recalling that $\Ext^\bullet(\Bsc^*, \cU^{0,0,-3}(-1))\cong \Ext^\bullet(\cU^{3,0,0}(1), \Bsc)=0$, because $\cU^{3,0,0}(1)\in \Bsc(1)$. Hence we have:
\begin{equation*}
\Ext^\bullet(\cF,\cE)=\Ext^\bullet(\cU^{2,0,-1}[-2],\,\cU^{0,0,-3}(-1)[3])=\Ext^\bullet(\cU^{2,0,-1},\,\cU^{0,0,-3}(-1))[5].
\end{equation*}
Finally, 
\begin{equation*}
\Ext^\bullet(\cU^{2,0,-1},\,\cU^{0,0,-3}(-1))\cong \H^\bullet(X, \cU^{1,0,-2}\otimes \cU^{0,0,-3}(-1))\cong \mathbb{C}[-5]
\end{equation*} by \cref{lem:tensor1}. 
\end{proof}

Let $\phi\in \Ext^{\bullet}(\cF,\cE)=\mathbb{C}$ be the nonzero morphism, which is unique up to scalar. Then, we can define $\cH$ as 
\begin{equation}\label{def:H}    \cH=\Cone(\,\cF\xrightarrow \phi \cE\,).
\end{equation}
 By \cite[Lemma 5.1]{guseva2020derived}, the morphism  $\phi:\cF\rightarrow \cE$ lifts uniquely to a morphism of the right resolutions of $\cF$ and $\cE$ (\eqref{eq:right F} and \eqref{eq:right E}), finally defining the maps in the bicomplex \eqref{H-intro}; hence $\cH$ defined by \eqref{def:H} is isomorphic to the totalization of \eqref{H-intro}.
 
As an immediate consequence, we obtain the following.

\begin{proposition}\label{excep}
The objects $\cE,\cF$ and $\cH$ are exceptional. Moreover, $\cH=\mathbb{L}_{\cF}\cE$.
\end{proposition}
\begin{proof}
We recall from \cref{subsec:except} that given an admissible subcategory $\Asc$, the mutation functor $\mathbb{L}_{\Asc}$ induces an equivalence between
${}^\perp \Asc$ and $\Asc^\perp$.
In particular, if $ G \in {}^\perp \Asc$ is  exceptional, then  $\mathbb{L}_{\Asc} G$ is exceptional as well. As $\bB_1$ and $\bB_2$ are exceptional collections by \cref{cor:bases}, the objects $\cU^{3,0,0}$ and $\cU^{2,0,-1}$ are exceptional and we have
\begin{equation*}
    \cU^{3,0,0}\in {}^\perp (\,\bS_1 \cap \bB_2\,) \quad \text{and} \quad \cU^{2,0,-1}\in {}^\perp (\,\bS_2\cap \bB_2\,),
\end{equation*} 
hence the objects $\cE$ and $\cF$ 
are exceptional by \cref{prop:mutate}.
Finally, as $\cE\in {}^\perp\cF$ by \cref{prop:mutate} and $\Ext^\bullet(\cF,\cE)\cong \mathbb{C}$ by \cref{prop:ext-isom}, we conclude that \begin{equation}
    \cH = \Cone{(\,\cF \xrightarrow \phi \cE\,)} = \mathbb{L}_{\cF} \cE
\end{equation} 
is exceptional.
\end{proof}

\begin{proposition}\label{prop:H is semiortho}
The object $\cH$ is right orthogonal to $\bB_1$, that is:
\begin{equation*}
    \Ext^\bullet(\bB_1, \cH)=0.
\end{equation*}
Hence, $\cH=\mathbb{L}_{\bB_1} \cU^{3,0,0}[3]$.
\end{proposition}
\begin{proof}
We have $\cH\in \bS^\perp$, because $\cE,\cF\in \bS^\perp$ by \cref{prop:mutate}. By \cref{excep}, we have $\cH\in \cF^\perp$ . But $\cU^{2,0,-1}\in \langle \, \cF, \bS \, \rangle$ by \eqref{eq:right F}, hence $\cH\in ( \, \cU^{2,0,-1} \, ) ^\perp$. Finally,  as we have $\bB_1=\bS \, \cup \, \{ \, \cU^{2,0,-1} \, \}$, we deduce that $\cH \in \bB_1^\perp$. This also proves $\cH = \mathbb{L}_{\bB_1} \cE$. By \cref{prop:mutate} we obtain the second claim.
\end{proof}

Finally we can state the main result of the section. Recall that $\bB = \{\,\cH \, \} \cup  \bB_1$. 
\begin{theorem}\label{thm:main}
The bounded derived category of coherent sheaves on $\IGr(3,9)$ admits a rectangular Lefschetz exceptional collection of length $8$ composed by $\Sp_9$-equivariant objects given by:
\begin{equation*}
    \langle \,  \bB,\,\bB(1),\,\bB(2),\,\bB(3),\,\bB(4),\,\bB(5),\,\bB(6)\, \rangle.
\end{equation*}
\end{theorem}
\begin{proof}
Since $\{\,\cH \, \} \,\cup\,  \bB_1 \subset \Bsc$, the blocks are semiorthogonal by \cref{prop:Semi-O}. On the other hand, semiorthogonality within a block follows from \cref{cor:bases} and \cref{prop:H is semiortho}. The object $\cH$ itself is exceptional by \cref{excep}. Applying \cref{lem:criterion-lef}, we deduce the exceptionality claim.

It is immediate to see that all the elements in $\bB_1$ admit an $\Sp_9$-equivariant structure. To show that $\cH$ admits one as well, we recall \cite[Lemma 2.2.(1)]{polishchuk2011k}, which proves that any exceptional object in $\IGr(k,2n+1)$ is $\Sp_{2n+1}$-equivariant. 

Recall that $\Sp_{2n+1}=(\mathbb{C}^*\times \Sp_{2n})\rtimes U$, by \cite[\textsection 3]{mihai2007odd}, where $\mathbb{C}^*\times \Sp_{2n}$ is the Levi subgroup,  and $U\cong \mathbb{C}^{2n}$  is the unipotent radical. As the Levi subgroup is reductive, we can apply \cite[Lemma 2.2.(2)]{polishchuk2011k} to $\mathbb{C}^*\times \Sp_{2n}$. Since $U\cong \mathbb{C}^{2n}$, we see that the hypothesis of \cite[Lemma 2.2.(1)]{polishchuk2011k} hold for $\Sp_{2n+1}$ as well. In particular, $\cH$ is $\Sp_9$-equivariant. 
\end{proof}

\begin{remark}\label{rem:vb}
    It is possible to show that $\cH$ is a vector bundle. To do so, we need to prove that the morphism $\phi$ defined by \cref{prop:ext-isom} is surjective. Approaching the problem in $\IGr(3,9)$ is hard, as $\IGr(3,9)$ is not homogeneous and $\Sp_9$ is not semisimple. 
    
    A lighter solution is to consider the embedding of $X$ in $\IGr(3,10)$. The staircase complexes of $\wcU^{3,0,0}$ and $\wcU^{2,0,-1}$ on $\IGr(3,10)$, appropriately truncated, provide lifts of $\cE$ and $\cF$. We can prove that the corresponding map is surjective, similarly to \cite[Lemma 5.3]{guseva2020derived}, then we can obtain the result on the original morphism applying \cref{prop:split-staircase}.
\end{remark}

\section{Fullness}\label{sec:full}
In \cite[\textsection~4-\textsection~5]{novikov}, it was developed a procedure to show the fullness of an exceptional collection on the even isotropic Grassmannian. We adapt it to the case of an odd isotropic Grassmannian and we summarize it in \cref{sec:strat}. This reduces the question of fullness to a combinatorial statement (\cf \cref{prop:final staircase}).

The procedure consists of consequent applications of steps of two different kinds to produce more and more objects in the category $\Dsc$ generated by the exceptional collection:
\begin{enumerate}
    \item using staircase complexes, already discussed in \cref{thm:staircase};
    \item using so-called "symplectic bundle relations", see \cref{sec:symp}.
\end{enumerate}
When we have produced sufficiently many objects, we can conclude that $\Dsc=\DX$ by \cref{prop:final staircase}. We give more details on this strategy in \cref{sec:strat} below and implement it to prove the fullness in \cref{sec:algorithmic} in nine steps.
From now on, we represent $\cU^{i,0,\dots,0,-j}$ with the shorthand notation
\begin{equation*}
    \cU^{i,-j}.
\end{equation*}

\subsection{Symplectic bundle relations}\label{sec:symp}

In this section we work with any odd isotropic Grassmannian $\IGr(k,V)$, with $\dimsf  V =2n + 1$. 

Consider the embedding $j:\IGr(k,V)\rightarrow \IGr(k,\widetilde{V})$ where $\wV$ is a symplectic vector space with $\dim \wV=2n+2$. Recall the notation fixed in \cref{subsec:odd-grass}. On the even isotropic Grassmannian $\IGr(k,\widetilde{V})$, the symplectic bundle $\widetilde{\cS}=\widetilde{\cU}^\perp/\, \widetilde{\cU}$ is the bundle of rank $2(n+1-k)$ defined as the cohomology in degree $0$ of the complex \eqref{eq:symplectic}. We consider the restriction to $\IGr(k,V)$ of the sequence \eqref{eq:symplectic}, that is:
\begin{equation}\label{eq:res}
    0\rightarrow\cU\rightarrow \widetilde{V}\otimes \cO\rightarrow \cU^* \rightarrow 0,
\end{equation}
which has only cohomology in degree $0$, which is isomorphic to $j^* \widetilde{\cS}$.

\begin{proposition}\label{prop:complex S}
    For any $p\geq 0$, the object $\wedge^p j^* \widetilde{\cS} \in \mathbf{D}^{\mathrm{b}}(\IGr(k,V))$ is quasi-isomorphic to a complex with entries given by direct sums, possibly with multiplicities, of $\cU^{i,-j}$ for $i,j\geq 0$ and $i+j\leq p$. Moreover, if $i+j=p$, the bundle $\cU^{i,-j}$ appears exactly once among the direct summands of the terms of this complex.
\end{proposition}
\begin{proof}
The complex obtained from the sequence \eqref{eq:res}, which we will denote as $\cC^\bullet$ is quasi-isomorphic to $j^*\widetilde{\cS}$.  Taking the wedge power of a complex preserves quasi-isomorphisms, hence $\wedge^p \cC^\bullet \cong \wedge^p j^* \widetilde{\cS}$ in $\mathbf{D}^{\mathrm{b}}(\IGr(k,V))$.
Recall that the monoidal structure of $\mathbf{D}^{\mathrm{b}}(\IGr(k,V))$ is defined in such a way that $\wedge^p (\cF[\pm 1])\cong (S^p \cF)[\pm p]$ for every vector bundle $\cF$.

The terms of the complex $\wedge^p \cC^\bullet$
are given by direct sums of:
\begin{equation}\label{eq:factors}
    \wedge^{p_1}(\cU^*[-1])\otimes \wedge^{p_2}(\widetilde{V}\otimes \cO) \otimes \wedge^{p_3} (\cU[1])
    = \cU^{p_1,0}[-p_1]\otimes \wedge^{p_2}(\widetilde{V}\otimes \cO) \otimes \cU^{0,-p_3}[p_3]
\end{equation}
with  
\begin{equation*}
    p_1+p_2+p_3=p, \quad \quad p_1, p_2, p_3 \geq 0
\end{equation*} 
By Pieri's formula, we can decompose the factors above as:
\begin{equation*}
    \cU^{p_1,0}\otimes \cU^{0,-p_3} = \bigoplus_{0\leq t \leq \min(p_1,p_3)} \cU^{p_1-t, t-p_3}.
\end{equation*}  
We now determine the multiplicity of $\cU^{i,-j}$ with $i+j=p$. If $\cU^{i,-j}$ comes from $\cU^{p_1,0}\otimes \cU^{0,-p_3}$, then   $i+j = (p_1-t)+(p_3-t)\leq p_1+p_3$. Then we obtain
$p_2=0, i=p_1, j=p_3$  and $t=0$.
This shows that $\cU^{i,-j}$ appears as a direct summand of a single term of $\wedge^p \cC^\bullet$. Moreover, its multiplicity is $\wedge^{p_2}\widetilde{V}\cong \mathbb{C}$, proving the claim.
\end{proof}

Restricting the symplectic isomorphism on $\IGr(k,\widetilde{V})$ we obtain $j^*\widetilde{\cS} \cong j^*\widetilde{\cS}^*$,
which induces:
\begin{equation}\label{eq:compare filters}
     \wedge^p j^*\widetilde{\cS} \cong \wedge^{2(n+1-k)-p}j^*\widetilde{\cS}^* \cong \wedge^{2(n+1-k)-p}j^*\widetilde{\cS}.
\end{equation}
Applying \cref{prop:complex S} to $\wedge^p j^*\widetilde{\cS}$ and $\wedge^{2(n+1-k)-p}j^*\widetilde{\cS}$, we obtain the following key proposition.

\begin{proposition}\label{prop:rule-S}
Let $\Dsc\subseteq\mathbf{D}^{\mathrm{b}}(\IGr(k,V))$ be a triangulated subcategory and let  $l\in \mathbb{Z}$. Let $i+j=p$ with $p> n+1-k$. If  $\cU^{i',-j'}(l)\in \Dsc$ for every $i',j'\geq 0$, with $i'+j'\leq p$ and $(i,j)\neq (i',j')$, then $\cU^{i,-j}(l)\in \Dsc$.
\end{proposition}
\begin{proof}
     Assume $l=0$. By hypothesis, we have $2(n+1-k) - p < n+1-k $. We first apply \cref{prop:complex S} to show that $\wedge^{2(n+1-k) - p}j^*{\widetilde{\cS}}\in \Dsc$. By \eqref{eq:compare filters}, that is equivalent to $\wedge^p j^*{\widetilde{\cS}}\in \Dsc$. Applying again \cref{prop:complex S}, $\wedge^p j^*{\widetilde{\cS}}$ admits a filtration with factors given by direct sums of $\cU^{i,-j}$ and $\cU^{i',-j'}$ as above. 
     
     By hypothesis, all the factors in the filtration belong to $\Dsc$ except possibly $\cU^{i,-j}$, which appears exactly once as a factor, hence $\cU^{i,-j}\in \Dsc$ as well. This proves the claim.  
     
     If $l\neq 0$, we apply the proven result to a twist of $\Dsc$.
\end{proof}

\subsection{Strategy}\label{sec:strat}
We show here how to reduce the proof of the fullness of the Lefschetz collection $\bB$ to an algorithmic question. Let us consider the case of $X=\IGr(3,2n+1)$. Consider the set of bundles 
\begin{equation}\label{def:T}
    \bT = \{\cU^{i,-j}\mid 0\leq i,j;\; i+j \leq 2n-2\}
\end{equation}

\begin{proposition}\label{prop:span class}
Let $\Dsc\subseteq \mathbf{D}^{\mathrm{b}}(X)$ be an admissible subcategory. If the set $\bT(-l)$ is contained in $\Dsc$ for every $l\geq 0$ then $\Dsc=\mathbf{D}^{\mathrm{b}}(X)$.
\end{proposition}
\begin{proof}
Since $\Dsc$ is admissible, consider the decomposition $\DX= \langle \Dsc^\perp, \Dsc \rangle$. Recall that $\cO(-l)$ for $l\geq 0$ is a spanning class (\cf \cite[Corollary~3.19]{huybrechts2006fourier}). This implies $ \Dsc^\perp = 0$ as $\cO(-l)\in \bT(-l)\subset \Dsc$, proving $\Dsc=\bD^{\mathrm{b}}(X)$.
\end{proof}

We rephrase here in a more compact form how to apply \cref{thm:staircase}.  
For each triple of integers $(a,b,c)$, with $a,c\geq -1$, $b\geq 0$ and $a+b+c\leq 2n-2$, we introduce a set of $a+b+c+2$ bundles named $\bS^{a,c}_b$ and defined as:
\begin{equation*}
    \bS^{a,c}_b :=\begin{cases*}
    \,\cU^{i,-b}  & \textrm{for} \quad $0\leq i\leq a$,\\
    \,\cU^{i,i-b+1}(-i-1) & \textrm{for} \quad  $0\leq i \leq b-1$,\\
    \,\cU^{b,-i}(-b-1) & \textrm{for} \quad
    $ 0\leq i \leq c$.
    \end{cases*}
\end{equation*}
It is immediate to verify that
\begin{equation}\label{eq:S-set-dual}
    (\bS^{a,c}_b)^*= \bS^{c,a}_b(b+1).
\end{equation}
We also remark that 
\begin{align}\label{eq:staircase-prop} 
    \bS^{a+1,c}_b= \{\,\cU^{a+1,-b}\,\}\sqcup \bS^{a,c}_b,\quad 
    \bS^{a,c+1}_b= \bS^{a,c}_b\sqcup \{\,\cU^{b,-c-1}(-b-1)\,\}.
\end{align} 

We now explain the relationship of $\bS_b^{a,c}$ with  staircase complexes. 

If $a, b\geq 0$ and $\maxsf\{a, b, a+b\} \leq 2n-2$, then we can apply \cref{thm:staircase} to compute the staircase complex of $\cU^{a,-b}$ on $\IGr(3,2n+1)$. As $a+b+c+2\leq 2n$, the collection $\bS_b^{a,c}$ is the set of the rightmost $a+b+c+2$ vector bundles appearing in this staircase complex (\cf \eqref{eq:weights}). If  $a+b+c = 2n-2$, then $\bS^{a,c}_b$ corresponds to the whole staircase complex.

We now consider the case where $a$ or $c$ is $-1$. Let us fix $a+b+c = 2n-3$, which corresponds to one term fewer than the amount of entries of a staircase complex on $\IGr(3, 2n+1)$. From the previous consideration and the identities stated in \eqref{eq:staircase-prop}, we obtain that
if $a,b\geq 0$ and $c = -1$, then $\bS_b^{a,-1}=\bS_b^{a,0}\setminus \;\{\,\cU^{b,0}(-b-1)\,\}$. That is, $\bS_b^{a,-1}$ is the set of entries of the staircase complex associated to $\cU^{a,-b}$ except $\cU^{b,0}(-b-1)$. 

Similarly, if $a=-1$, $2n-2\geq b\geq0$ and $c \geq 0$, then ${\bS_b^{-1,c}=\bS_b^{0,c}\setminus\; \{\,\cU^{0,-b}\,\}}$, which corresponds to all the entries of the staircase complex associated to $\cU^{0,-b}$,  except $\cU^{0,-b}$ itself. 

\begin{proposition}\label{prop:staircase compact}
Let $\Dsc\subseteq \mathbf{D}^{\mathrm{b}}(\IGr(3,2n+1))$ be a full triangulated subcategory and $l\in \mathbb{Z}$. Let $(a,b,c)$ be a triple with $a+b+c=2n-3$ and $a,c \geq -1$ and $2n-2\geq b\geq 0$. If $\bS^{a,c}_b(l)\subset \Dsc$, then we have the following cases:
\begin{itemize}
    \item if $a\geq-1$ and $b,c \geq0$, then $ \cU^{a+1,-b}(l)\in \Dsc$;
    \item  if $a,b\geq 0$ and $c\geq -1$, then $\cU^{b,-c-1}(l-b-1)\in \Dsc$.
\end{itemize}
\end{proposition}

\begin{proof}
    In the first case, we consider the staircase complex of $\cU^{a+1,-b}(l)$ as the left resolution of its rightmost term. The terms in the resolution are multiples of the bundles in  $\bS^{a,c}_b(l)\subset \Dsc$, we conclude that $\cU^{a+1,-b}(l)\in \Dsc$. 
    
    For the second statement, we apply \eqref{eq:S-set-dual}. From the hypotheses, we have that $(\, \bS_b^{a,c}(l)\,)^*=\bS_b^{c,a}(b+1-l)\subset \Dsc^*$. Applying the first statement, we obtain $\cU^{c+1,-b}(b+1-l)\in \Dsc^*$, hence $\cU^{b,-c-1}(l-b-1)\in \Dsc$, proving the second claim.
\end{proof}

The following \cref{prev} formalizes the idea of applying the staircase complex several times in a row. 
Consider the following partitions of the set $\bT$ defined in \eqref{def:T}:
\begin{equation}\label{eq:partit}
    \bT=\bigsqcup_{0\leq b \leq 2n-2} \bP_b =
    \bigsqcup_{0\leq a \leq 2n-2} \bN_a.
\end{equation}
where
\begin{equation*}
    \bP_b=\{\,\cU^{i,-b}\mid 0\leq i \leq 2n-2-b\,\},\quad
    \bN_a=\{\,\cU^{a,-i}\mid 0\leq i \leq 2n-2-a\,\}.
\end{equation*}

Let $\Dsc$ be a triangulated subcategory. Fixing $b$, we observe that the  staircase complexes associated to $\cU^{i,-b}$ contains the left resolution of $\cU^{i+1,-b}$. This allows us to generate $\cU^{i+1,-b}$ starting from the staircase complex of  $\cU^{i,-b}$. Proceeding inductively, this allows to generate all the bundles $\cU^{j,b}$ with $j\geq i$, which are all contained in $\bP_b$. The same works with the right ends of the staircase complexes and $\bN_a$.

This fact is shown in the following picture (\cref{explanatory}), where we represent the weights in $\bT$ for $n=4$. In this picture, the sets $\bP_i$ are given by the ascending diagonals with second coordinate equal to $i$, while the sets $\bN_i$ are given by descending diagonals with first coordinate $i$. 

We circle the elements of $\bS_{1}^{5,0}$, $\bS_{1}^{4,1}$ and $\bS_{1}^{3,2}$, which correspond to  the bundles belonging to three staircase complexes; notice that they have several elements in common. This shows that  once we proved that $\bS_{1}^{2,2}\subset \Dsc$, then $\bS_{1}^{3,2}\subset \Dsc$ by \cref{prop:staircase compact}, but then,  the left part of the staircase complex of $\cU^{4,-1}$ is contained in $\bS_{1}^{3,2}$, so we can apply repeatedly  \cref{prop:staircase compact}.

\begin{figure}[H]
\begin{minipage}[c]{0.5\textwidth}
\begin{center}
\includegraphics[scale=0.5]{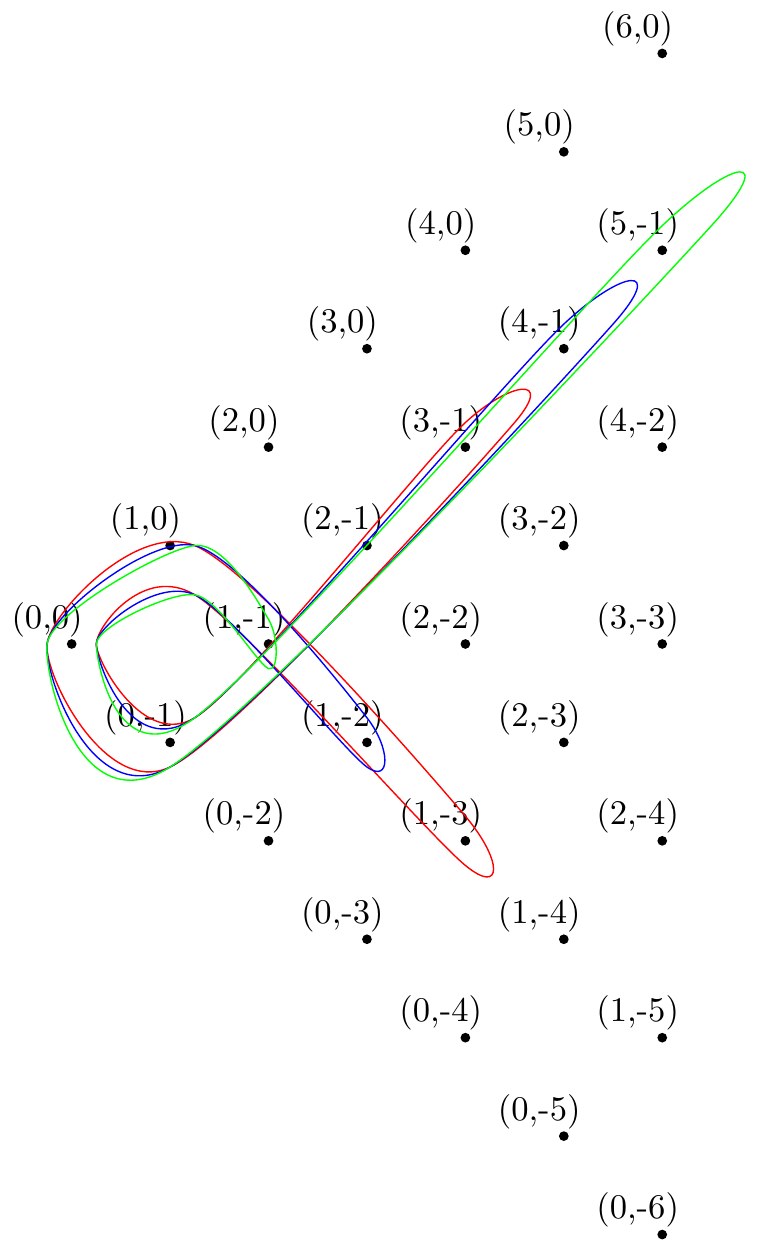} 
\end{center}
\end{minipage}\hfill
\begin{minipage}[c]{0.5\textwidth}
\begin{center}
\caption{The weights in $\bT$ for $n=4$. The sets   $\bS_1^{5,0}$, $\bS_1^{4,2}$, $\bS_1^{3,3}$ are in green, blue, red, respectively.}
\label{explanatory}
\end{center}
\end{minipage}
\end{figure}

Using this observation, we prove the following.

\begin{proposition}\label{prev}
Let $\Dsc\subseteq \mathbf{D}^{\mathrm{b}}(\IGr(3,2n+1))$ be a triangulated subcategory and $l\in \mathbb{Z}$. Let $(a,b,c)$ be a triple with $a+b+c = 2n-3$ and $a,c \geq -1$ and $2n-2\geq b\geq 0$. If $\bS^{a,c}_b(l)\subset \Dsc$, then:
\begin{equation*}
    \bP_b(l)\subset \Dsc \quad and \quad \bN_b(l-b-1) \subset \Dsc.
\end{equation*}
\end{proposition}
\begin{proof}
Assume $l=0$. We prove the statement for any triple $(a,b,c)$ with the properties stated above, in particular, $a=2n-3-c-b$. We prove $\bP_b\subset \Dsc$, proceeding by induction on $c$. If $c=-1$, we have $a=2n-2-b$. By definition:
\begin{equation*}
    \bS^{a,-1}_b=\{\,\cU^{i,-b} \mid 0\leq i\leq a\,\}\,\sqcup\, \{\,\cU^{b,-i}(-i-1) \mid 0\leq i\leq b-1\,\} \supseteq \bP_b,
\end{equation*}
proving the base of induction. Suppose $c\geq 0$, then by \cref{prop:staircase compact} we have $\cU^{a+1,-b}\in \Dsc$. Applying \eqref{eq:staircase-prop}, we have $\bS^{a+1,c}_b\subset \Dsc$. Finally with \eqref{eq:staircase-prop}, $\bS^{a+1,c-1}_b \subset \bS^{a+1,c}_b\subset \Dsc$. By the induction hypothesis, $\bP_b\subset \Dsc$, proving the first part of the statement. When $l\neq 0$, we apply the result to a twist of $\Dsc$.

From the statement on $\bP_b$ we deduce the result for $\bN_b$ by dualizing and applying \eqref{eq:S-set-dual} (as in the final part of \cref{prop:staircase compact}).
\end{proof}

\begin{proposition}\label{prop:final staircase}
Let $\Dsc\subseteq \DX$ be a full admissible subcategory. If $\bT(l)\subset \Dsc$ for ${l=0,\dots, 2n-2}$, then $\Dsc=\DX$.
\end{proposition}
\begin{proof}
Let us consider the set 
\begin{align*}
    \bS^{2n-2-b,-1}_b &= \{\,\cU^{i,-b} \mid 0\leq i\leq 2n-2-b\,\}\,\sqcup\, \{\,\cU^{b,-i}(-i-1) \mid 0\leq i\leq b-1\,\}\subseteq \\
    &\subseteq \bT \,\sqcup\, \bigsqcup_{0\leq i\leq b-1} \bT(-i-1).
\end{align*}
In particular, twisting the previous inclusion by $\cO(b)$, for $0\leq b\leq 2n-2$, we obtain:
\begin{equation*}
    \bS^{2n-2-b,-1}_b(b)\subseteq \bT(b) \sqcup \bigsqcup_{0\leq i\leq b-1} \bT(b-i-1)\subset \Dsc.
\end{equation*}
By \cref{prev}, we have $\bN_b(-1)\subset \Dsc$ for any $0\leq b\leq 2n-2$. As $\bT(-1)$ can be partitioned as the disjoint union of $\bN_b(-1)$ for $0\leq b \leq 2n-2$ by \eqref{eq:partit},  we obtain $\bT(-1)\subset \Dsc$. Inductively, we obtain $\bT(-l)\subset \Dsc$ for every $l\geq 0$, concluding the proof with \cref{prop:span class}. 
\end{proof}

Finally, we are able to show the fullness in the Lefschetz collection with basis $\bB$ on $\IGr(3,9)$.
\subsection{Proof of fullness}\label{sec:algorithmic}

We go back to the case $X=\IGr(3,9)$. Let 
\begin{equation*}
    \Dsc:= \langle \,  \bB,\,\bB(1),\,\bB(2),\,\bB(3),\,\bB(4),\,\bB(5),\,\bB(6)\, \rangle\subseteq \DX.
\end{equation*}
By \cref{thm:main}, $\Dsc$ is generated by an exceptional collection, hence it is admissible (\cf \cref{subsec:except}). We want to prove that $\Dsc=\DX$.
Recall that
\begin{equation}\label{start}
    \bB_1\cup \bB_2 = \{\,\cU^{0,0,-2}, \, \cU^{0,0,-1}, \, \cU^{1,0,-1} ,\, \cU^{2,0,-1},\, \cU^{0,0,0} ,\, \cU^{1,0,0} ,\, \cU^{2,0,0},\, \cU^{3,0,0}\,\}\subset \langle \, \bB \, \rangle,
\end{equation}
hence  $(\, \bB_1\cup \bB_2\, )(l)\subset \Dsc$ for $l=0,\dots,6$. 
To prove that $\Dsc$ is full, we go through the following algorithm, which consists in the iteration of these two steps:
\begin{itemize}
    \item determine all the $\bS^{a,c}_b(l)\subset \Dsc$ with $a+b+c = 6$ and apply \cref{prev};
    \item  apply \cref{prop:rule-S} where is possible.
\end{itemize}
The procedure terminates if we prove that the following set of bundles 
\begin{equation*}
    \bT,\,\bT(1),\,\dots,\, \bT(6),
\end{equation*}
where $\bT$ was defined in  \eqref{def:T}, is contained in $\Dsc$. Then we can apply \cref{prop:final staircase} and prove fullness. Notice that it is not certain that the presented algorithm  terminates, even if the starting collection is full.
In our specific case, we are able to show in nine steps that the condition in \cref{prop:final staircase} holds. 
\begin{theorem}\label{prop:fullness}
    The condition in \cref{prop:final staircase} holds for $X=\IGr(3,9)$ and the following admissible subcategory 
    \begin{equation*}
    \Dsc := \langle \,  \bB,\,\bB(1),\,\bB(2),\,\bB(3),\,\bB(4),\,\bB(5),\,\bB(6)\, \rangle.
\end{equation*}
Hence, $\Dsc=\DX$.
\end{theorem}
For clarity of exposition we present the steps in this procedure next to diagrams to illustrate the progress in the proof. We apply a label $\boxed{l_1\olddiv l_2}$ next to every colored area to denote for which twists we know already that the bundles in the area belong to $\Dsc$.
\begin{figure}[H]
\begin{minipage}[c]{0.5\textwidth}
\begin{center}
\includegraphics[scale=0.55]{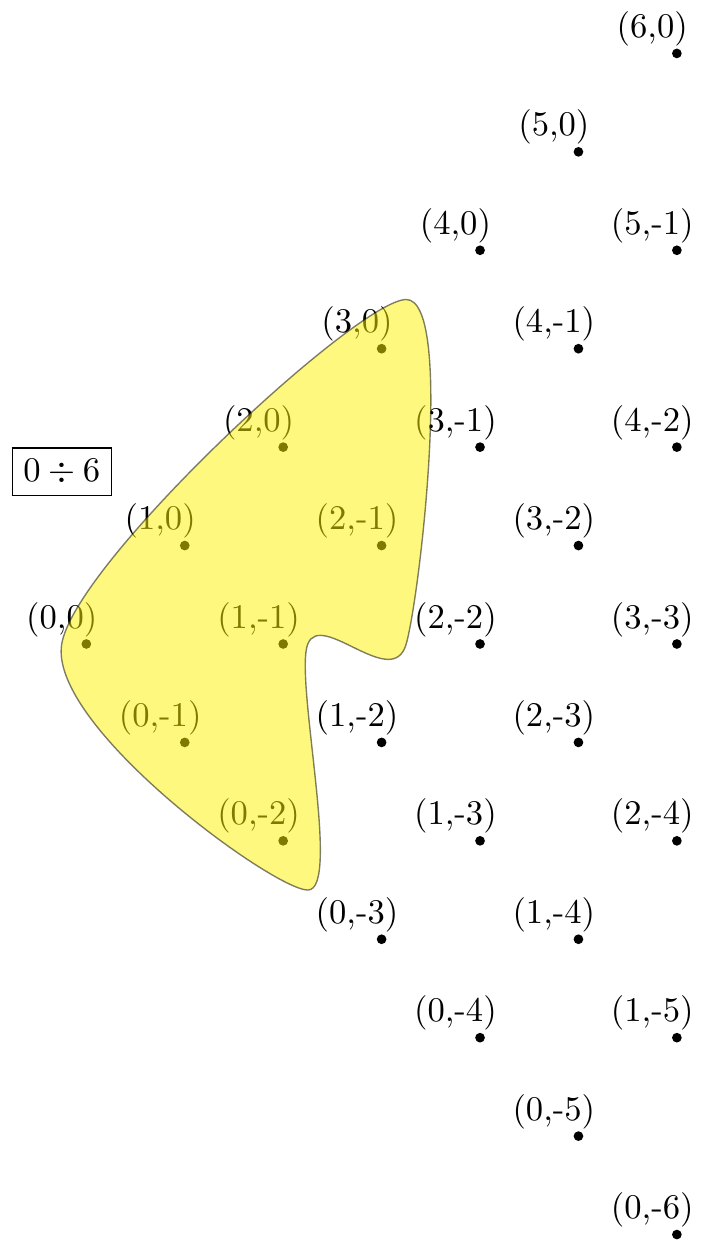}
\end{center}
\end{minipage}\hfill
\begin{minipage}[c]{0.5\textwidth}
\caption{In this triangular diagram, we represent the weights of the bundles of $\IGr(3,9)$ belonging to $\bT$. In yellow, the bundles in $\bB_1\cup \bB_2$.}
\end{minipage}
\end{figure}
\newpage
\subsubsection{Step 1.}
\begin{figure}[!ht]
\begin{center}
\includegraphics[scale=0.65]{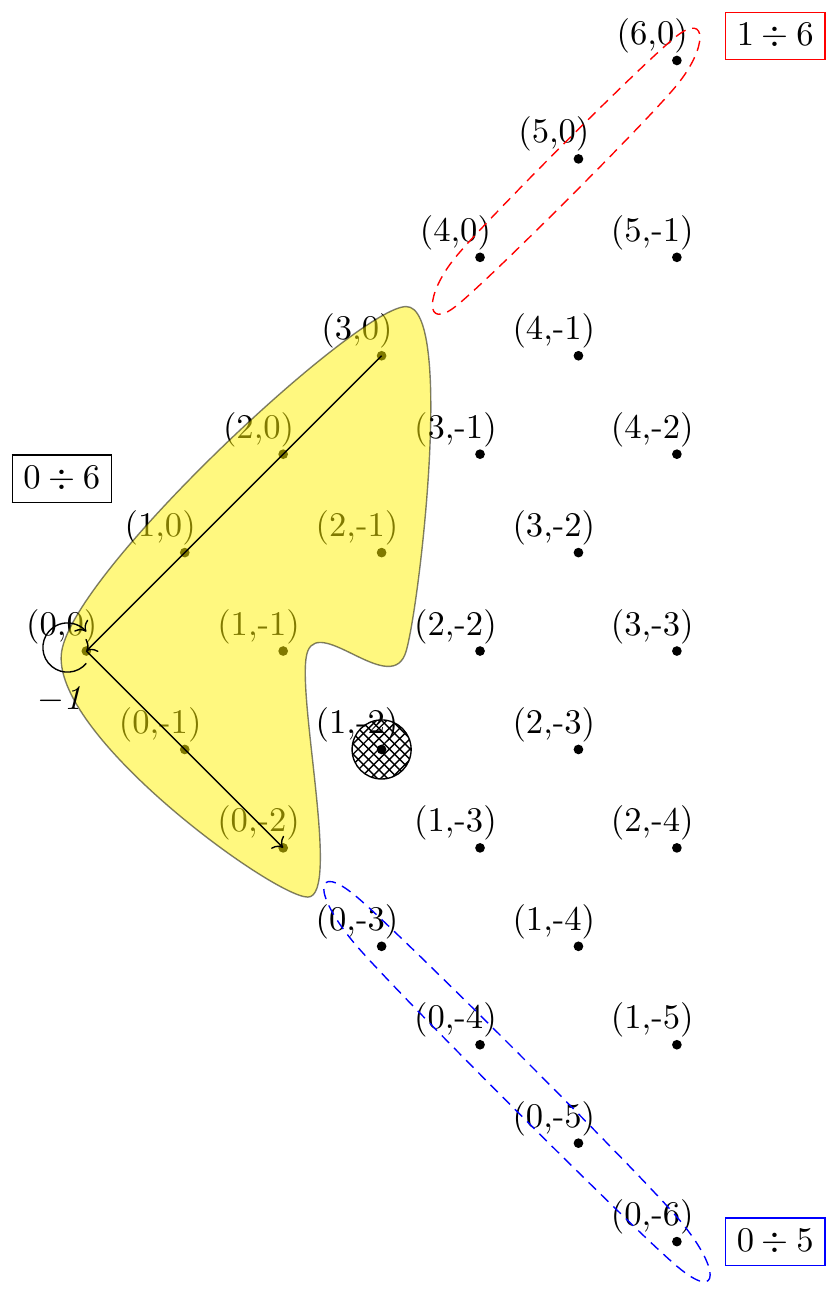}
\end{center}
\caption{The set $\bS^{3,2}_0$ and the bundle $\cU^{1,-2}$.}
\end{figure}
  As $\bS_0^{3,2}(l)\subset \Dsc$ for $l=1,\dots,6$, using \cref{prev} we obtain:
\begin{align*}
    \bP_0(l) &\subset \Dsc,  &\textrm{for} \quad l=1,\dots,6, \\
    \bN_0(l) &\subset \Dsc,  &\textrm{for} \quad l=0,\dots,5.
\end{align*}
Moreover, we notice that: 
\begin{align*}
    \cU^{i,-j}(l)\in \Dsc \quad &\textrm{for} \quad l=0,\dots,5\\
    &\textrm{and} \quad i+j\leq 3,
\end{align*}
except for $\cU^{1,-2}(l)$. Applying \cref{prop:rule-S}, we conclude
\begin{equation*}
    \cU^{1,-2}(l)\in \Dsc \quad \textrm{for} \quad l=0,\dots,5.
\end{equation*}

\newpage
\subsubsection{Step 2.}
\begin{figure}[!ht]
\begin{center}
\includegraphics[scale=0.7]{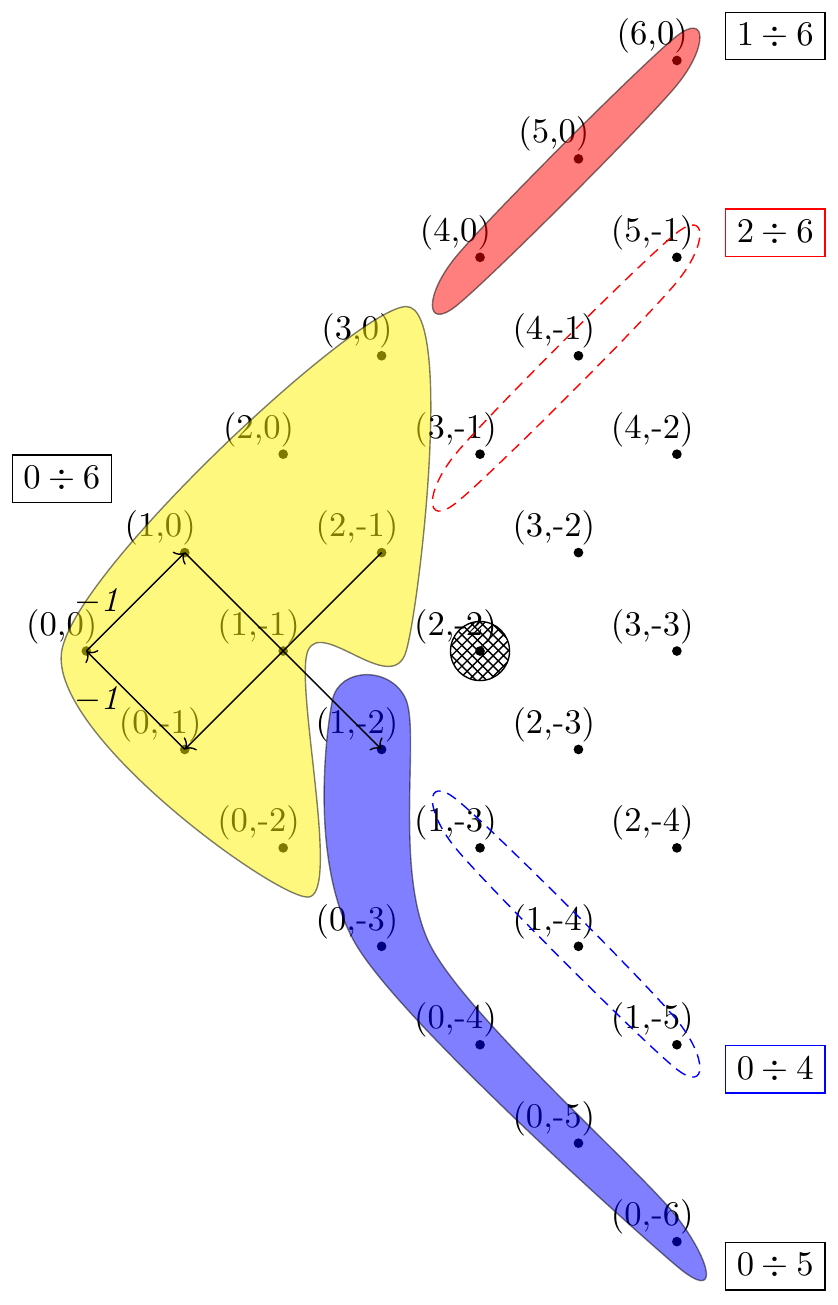}
\end{center}
\caption{The set $\bS_1^{2,2}$ and $\cU^{2,-2}$.}
\end{figure}
Given that $\bS_1^{2,2}(l)\subset \Dsc$ for $l=2,\dots,6$, we show:
\begin{align*}
    \bP_1(l) &\subset \Dsc,  &\textrm{for} \quad l=2,\dots,6, \\
    \bN_1(l) &\subset \Dsc,  &\textrm{for} \quad l=0,\dots,4.
\end{align*}
Applying \cref{prop:rule-S} to $\cU^{2,-2}$, we obtain that 
\begin{equation*}
    \cU^{2,-2}(l)\in \Dsc \quad \textrm{for} \quad l=2,\dots,4.
\end{equation*}
\newpage
\subsubsection{Step 3.}
\begin{figure}[!ht]
\begin{center}
\includegraphics[scale=0.7]{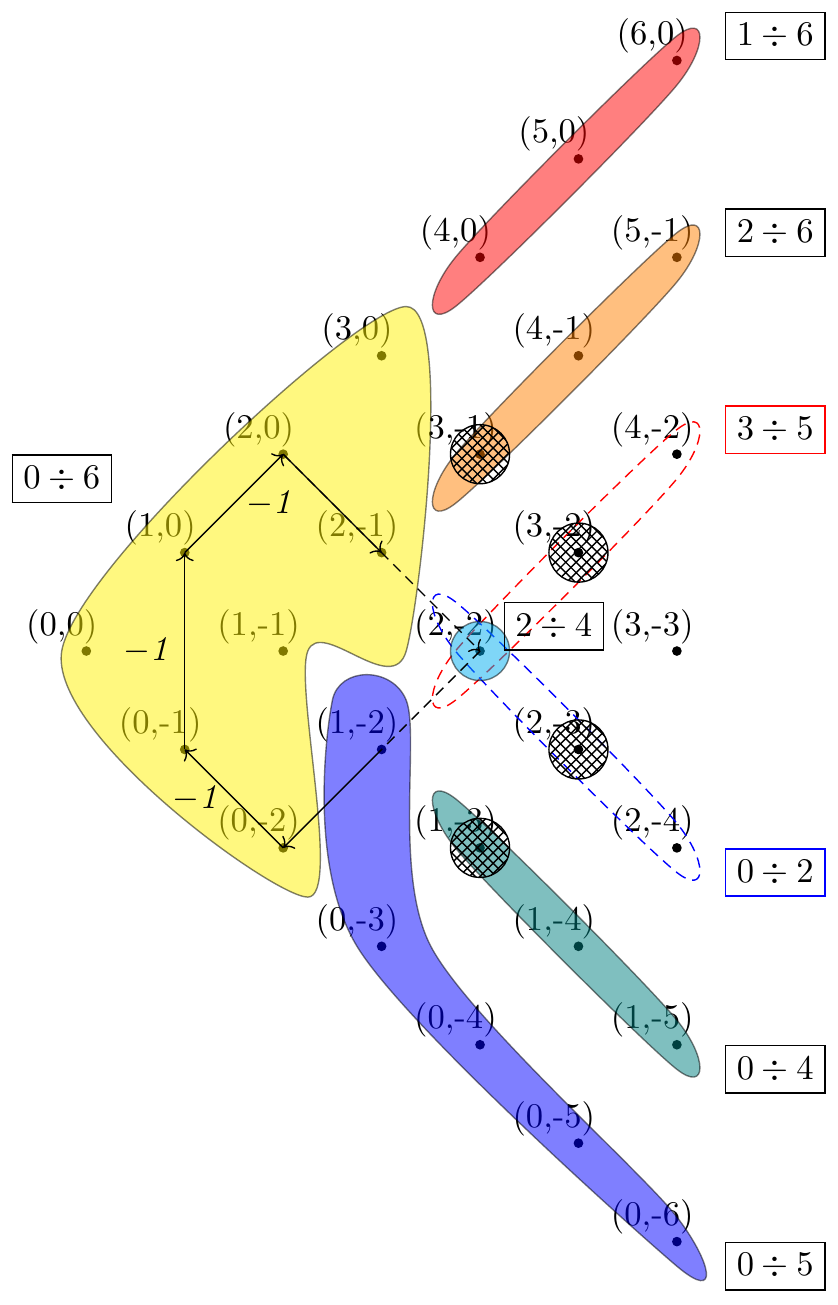}
\end{center}
\caption{In solid black arrows, $\bS_2^{1,2}\cap \bS_2^{2,1}$. Inside the circles, the bundles obtained applying \cref{prop:rule-S}.}
\end{figure}
As $\bS_2^{1,2}(l)\subset \Dsc$ for $l=5$ and $\bS_2^{2,1}(l)\subset \Dsc$ for $l=3,4$, we have:
\begin{align*}
    \bP_2(l) &\subset \Dsc, & &\textrm{for} \quad l=3,\dots,5, \\
    \bN_2(l) &\subset \Dsc, & &\textrm{for} \quad l=0,\dots,2.
\end{align*}
Applying \cref{prop:rule-S}, we obtain:
\begin{align*}
    \cU^{3,-1}(1) &\in \Dsc, \\
    \cU^{1,-3}(5) &\in \Dsc, \\
    \cU^{3,-2}(2) &\in \Dsc, \\
    \cU^{2,-3}(3),\, \cU^{2,-3}(4) &\in \Dsc.
\end{align*}

\newpage

\subsubsection{Step 4.}
\begin{figure}[!ht]
\begin{center}
\includegraphics[scale=0.7]{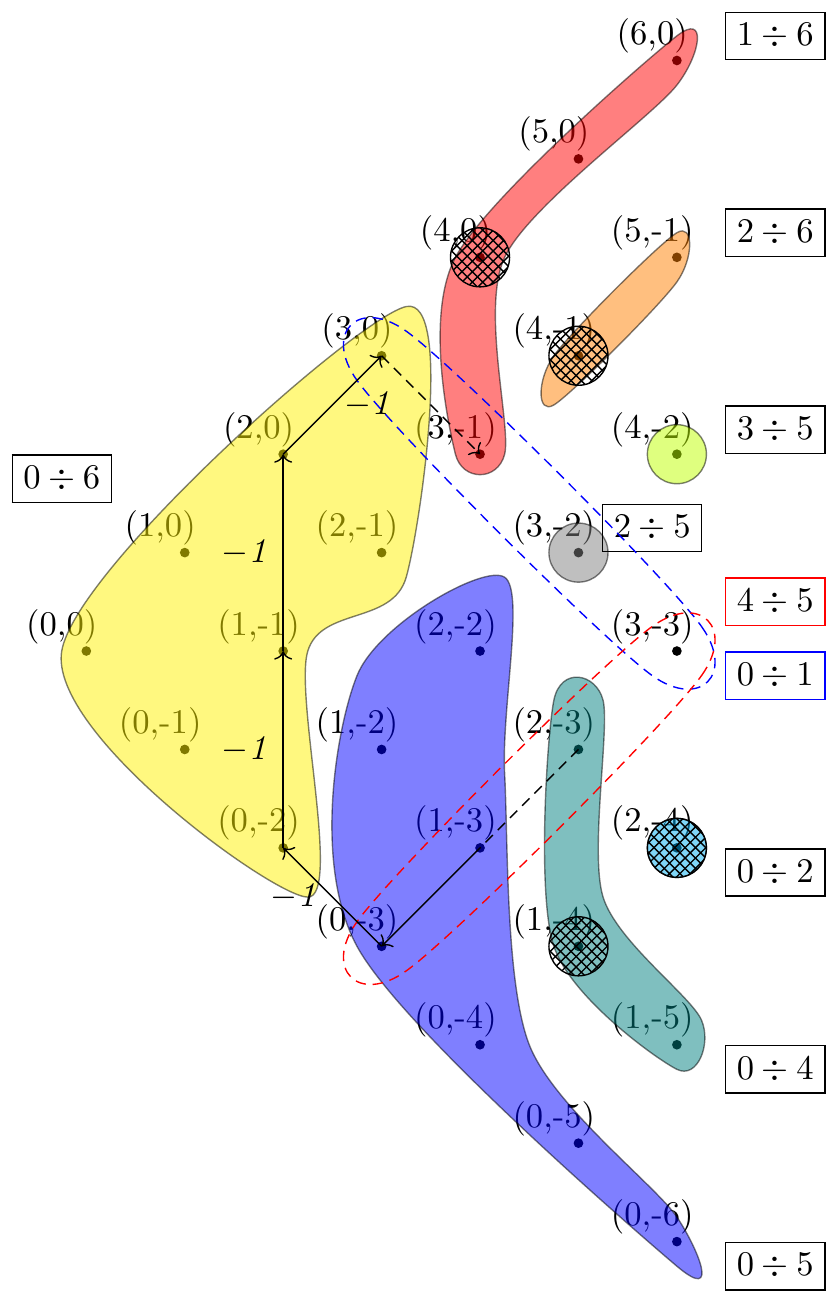}
\end{center}
\caption{In continuous black arrows, $\bS_3^{2,0}\cap \bS_3^{1,1}$. Inside the circles, the bundles obtained applying \cref{prop:rule-S}.}
\end{figure}
As $\bS_3^{2,0}(4)\subset \Dsc$ and $\bS_3^{1,1}(5)\subset \Dsc$, we have:
\begin{align*}
    \bP_3(l) &\subset \Dsc, & &\textrm{for} \quad l=4,5, \\
    \bN_3(l) &\subset \Dsc, & &\textrm{for} \quad l=0,1.
\end{align*}
Applying \cref{prop:rule-S}, we obtain the following bundles:
\begin{align*}
    \cU^{4,0} &\in \Dsc,\\
    \cU^{4,-1}(1) &\in \Dsc,\\
    \cU^{1,-4}(5) &\in \Dsc,\\
    \cU^{2,-4}(4) &\in \Dsc.
\end{align*}

\newpage
\subsubsection{Step 5.}
\begin{figure}[!ht]
\begin{center}
\includegraphics[scale=0.7]{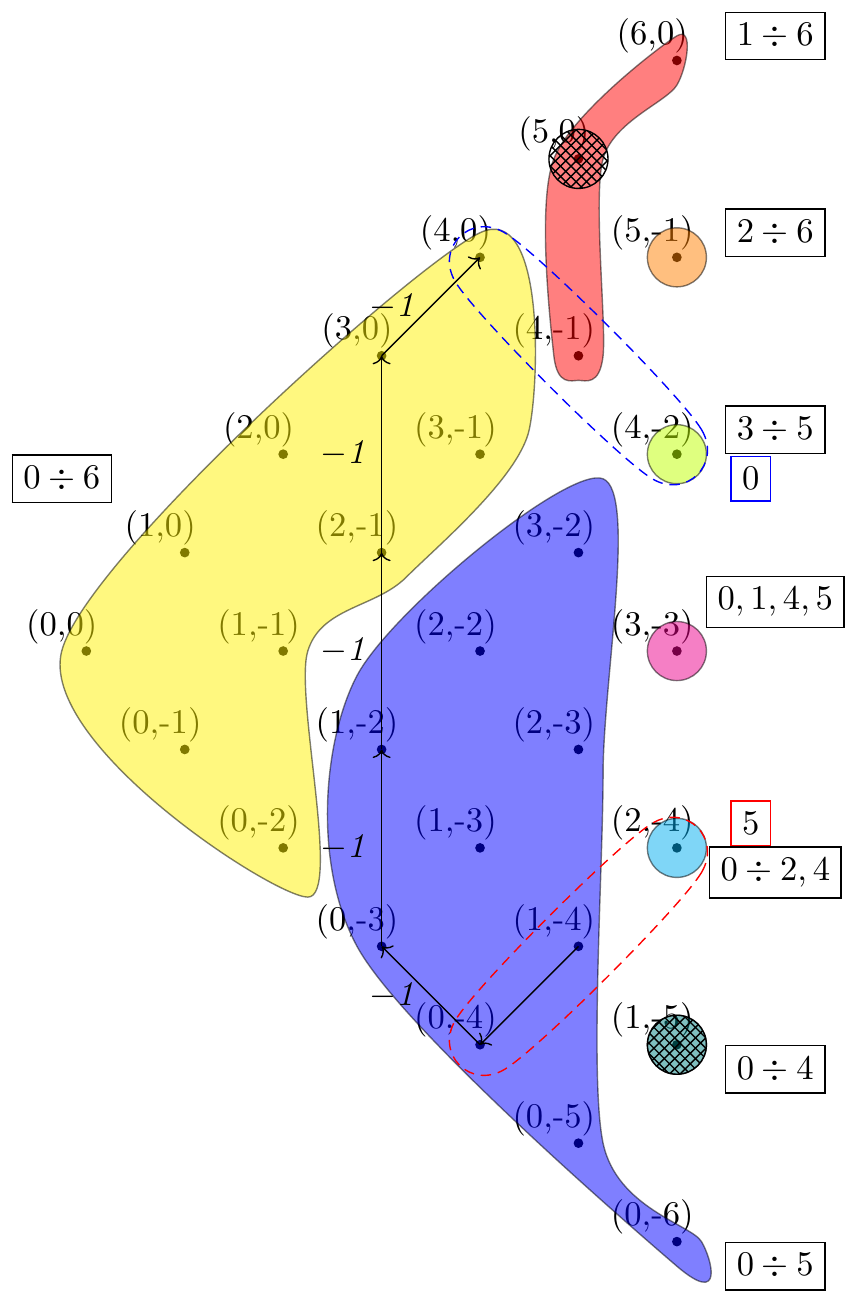}
\end{center}
\caption{The set $\bS_4^{1,0}$ and the bundles obtained by applying \cref{prop:rule-S}.}
\end{figure}
As $\bS_4^{1,0}(5)\subset \Dsc$, we obtain:
\begin{align*}
    \bP_4(5) &\subset \Dsc,\\
    \bN_4(0) &\subset \Dsc.
\end{align*}
Applying \cref{prop:rule-S}, we obtain:
\begin{align*}
    \cU^{5,0} &\in \Dsc,\\
    \cU^{1,-5}(5) &\in \Dsc.
\end{align*}

\newpage 
\subsubsection{Step 6.}
\begin{figure}[!ht]
\begin{center}
\includegraphics[scale=0.7]{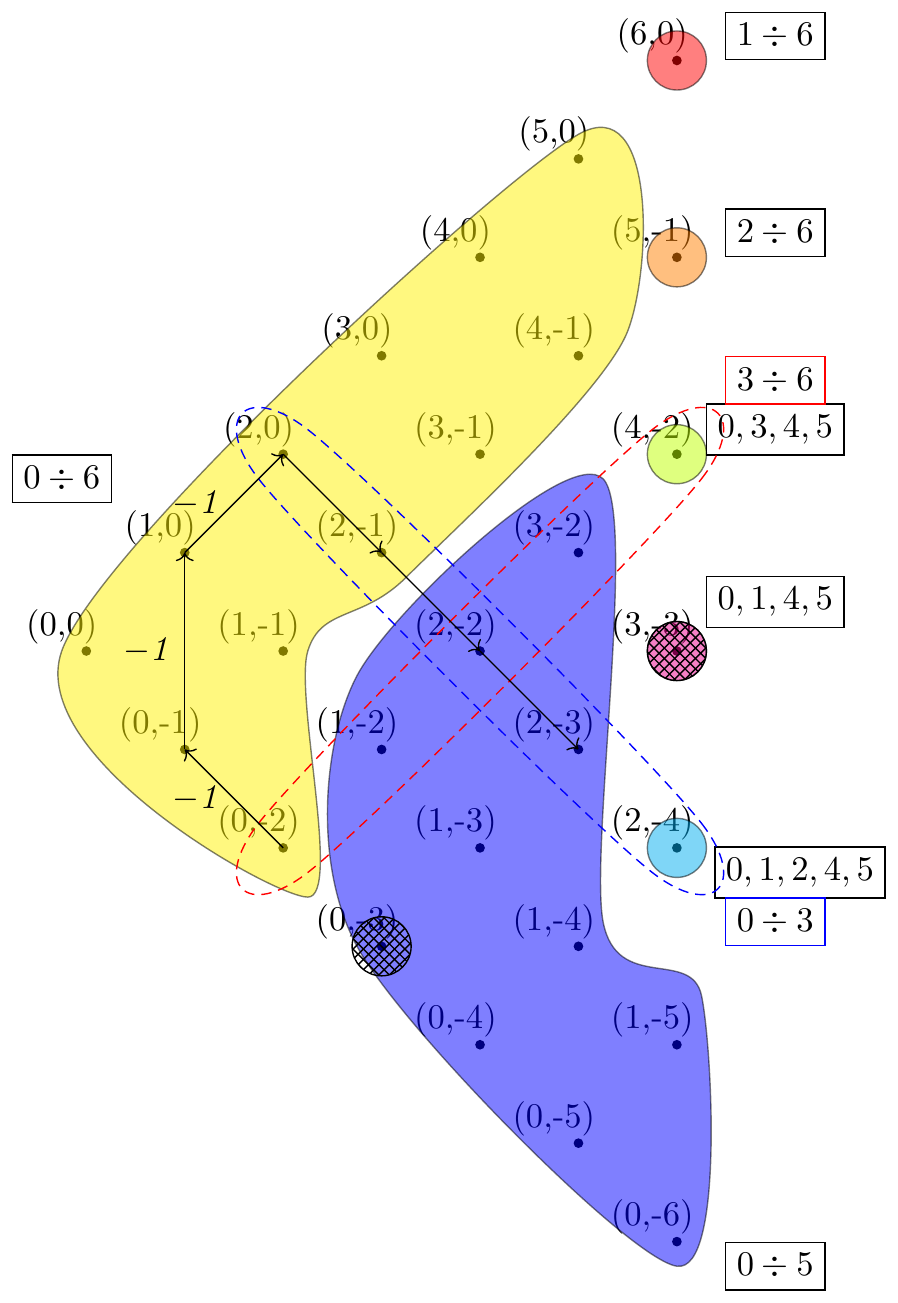}
\end{center}
\caption{The set $\bS_2^{0,3}$ and the bundles obtained by applying \cref{prop:rule-S}.}
\end{figure}
As $\bS_2^{0,3}(6)\subset \Dsc$, we have:
\begin{align*}
    \bP_2(6) &\subset \Dsc,\\
    \bN_2(3) &\subset \Dsc.
\end{align*}
Applying \cref{prop:rule-S} for $i+j=3$ and $l=6$, we obtain
\begin{equation*}
    \cU^{0,-3}(6) \in \Dsc;
\end{equation*}
while using $i+j=6$ and $l=3$, we obtain
\begin{equation*}
    \cU^{3,-3}(3) \in \Dsc.
\end{equation*}
\newpage
\subsubsection{Step 7.}
\begin{figure}[!ht]
\begin{center}
\includegraphics[scale=0.7]{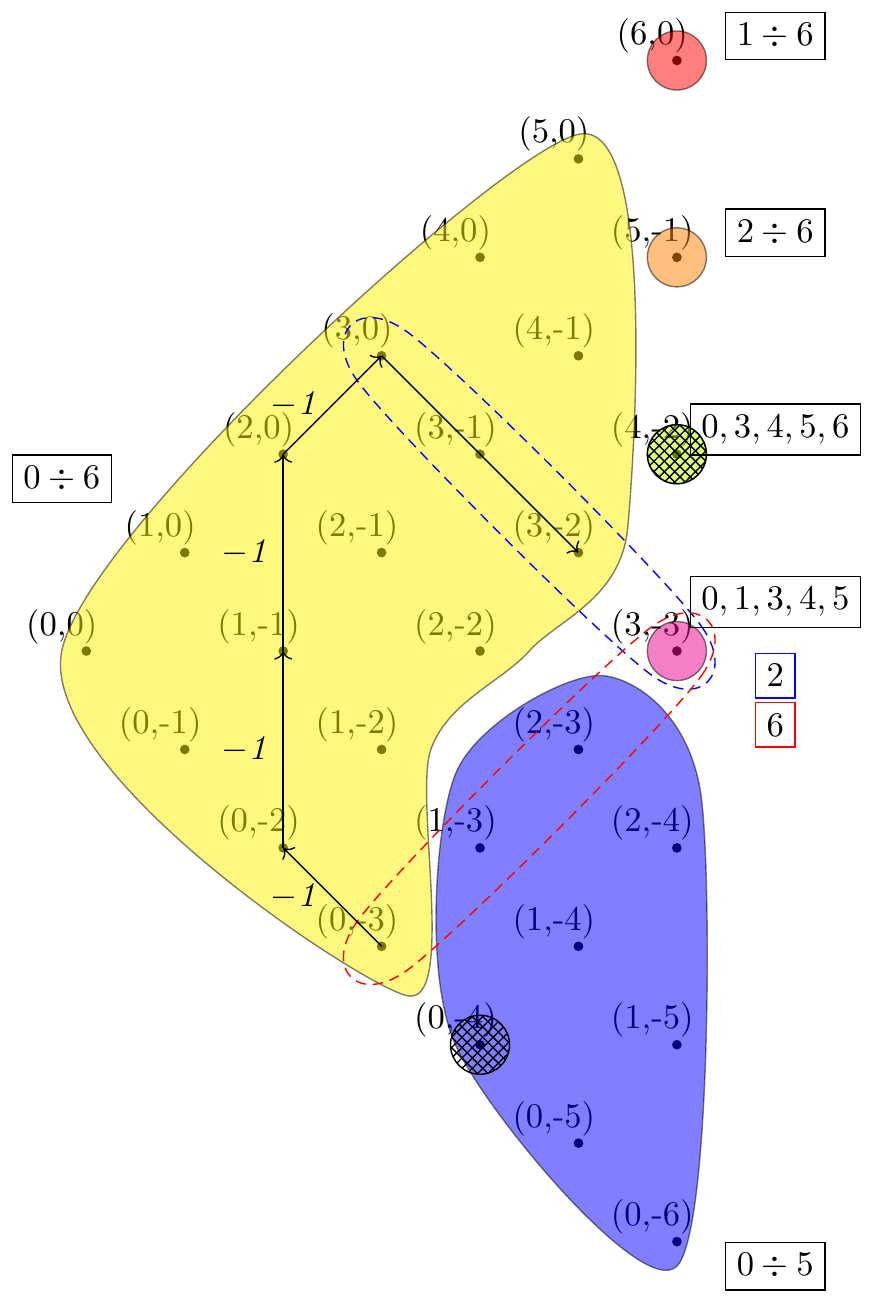}
\end{center}
\caption{The set $\bS_3^{0,2}$ and the bundles obtained by applying \cref{prop:rule-S}.}
\end{figure}
As $\bS_3^{0,2}(6)\subset \Dsc$, we have:
\begin{align*}
    \bP_3(6) &\subset \Dsc,\\
    \bN_3(2) &\subset \Dsc.
\end{align*}
Applying \cref{prop:rule-S}, we obtain:
\begin{align*}
    \cU^{0,-4}(6) &\in \Dsc,\\
    \cU^{4,-2}(2) &\in \Dsc.
\end{align*}
\newpage
\subsubsection{Step 8.}
\begin{figure}[!ht]
\begin{center}
\includegraphics[scale=0.7]{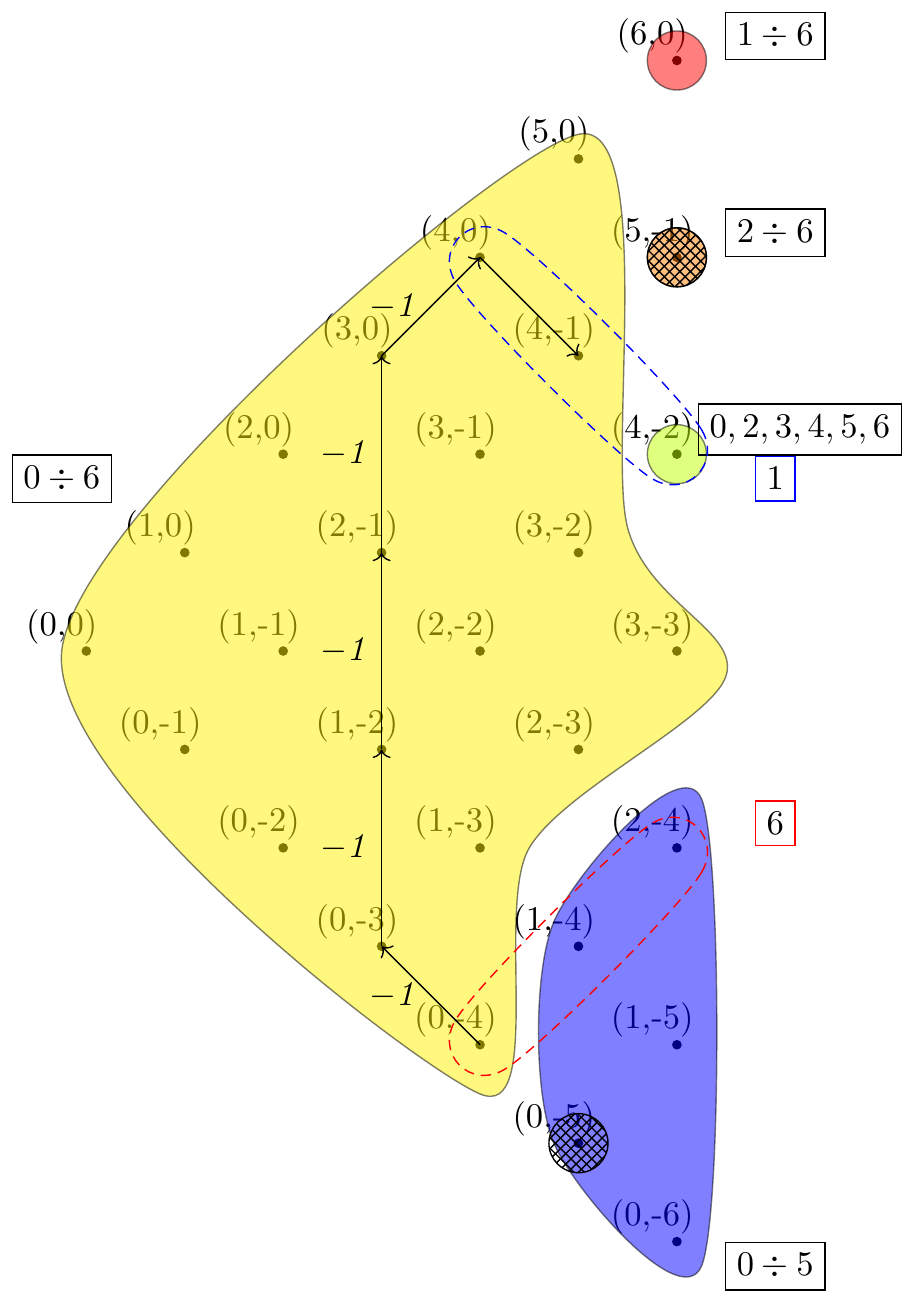}
\end{center}
\caption{The set $\bS_4^{0,1}$ and the bundles obtained by applying \cref{prop:rule-S}.}
\end{figure}
As $\bS_4^{0,1}(6)\subset \Dsc$, we have:
\begin{align*}
    \bP_4(6) &\subset \Dsc,\\
    \bN_4(1) &\subset \Dsc.
\end{align*}
As all bundles $\cU^{i,-j}(6)\in \Dsc$ for $i+j\leq 5$, applying \cref{prop:rule-S}, we obtain:
\begin{equation*}
    \cU^{0,-5}(6) \in \Dsc;
\end{equation*}
analogously, for $\cU^{i,-j}(1)\in \Dsc$ for $i+j\leq 6$, we obtain:
\begin{equation*}
    \cU^{1,-5}(1) \in \Dsc.
\end{equation*}
\newpage
\subsubsection{Step 9.}
\begin{figure}[!ht]
\begin{center}
\includegraphics[scale=0.7]{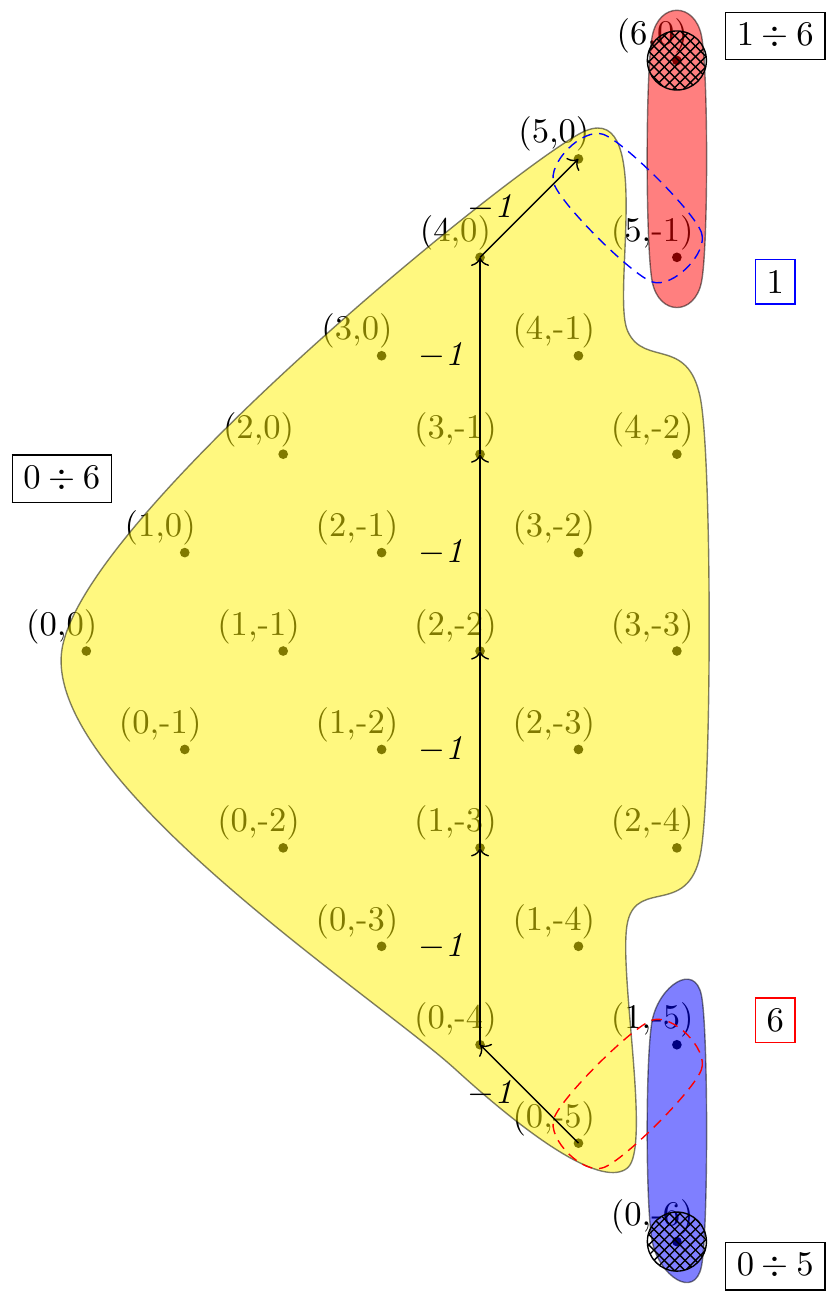}
\end{center}
\caption{The set $\bS_5^{0,0}$ and the remaining bundles.}
\end{figure}
As $\bS_5^{0,0}(6)\subset \Dsc$, we have:
\begin{align*}
    \bP_5(6) &\subset \Dsc,\\
    \bN_5(0) &\subset \Dsc.
\end{align*}
Applying \cref{prop:rule-S}, we finally generate the remaining bundles. 

\subsubsection{Conclusion}
After these nine iterations, we have shown that
\begin{equation*}
    \bT,\,\bT(1),\,\dots, \bT(6)\subset \Dsc.
\end{equation*}
Applying \cref{prop:final staircase} to $\Dsc$, we show that the Lefschetz collection induced   by  $\bB$ is full, finally  proving \cref{prop:fullness}. 
\appendix
\newpage

\section{Grothendieck group of horospherical varieties of Picard rank one}\label{sec:appendix}

Recall the notation introduced in \cref{thm:pas}. The following \cref{prop:invariants} shows that the Grothendieck group $K_0(X)$ of a horospherical variety $X$ is determined by the Grothendieck groups of its homogeneous pieces $Y,Z$. In particular, this addresses the case of odd isotropic Grassmannians. A similar results has been proved previously in singular cohomology (\cf \cite[Fact~1.8]{gonzales2018geometry}), with similar ideas, but the details of the proof were left to the reader. We provide a complete proof here.

For additional details on the following construction, we refer to \cite{pasquier2009some} and to  \cite[\textsection~1.5]{gonzales2018geometry}. We remark that there is an asymmetry in the choice of $Z$ and $Y$, as one of the closed orbits is stable under the action of the non-reductive group $\Aut(X)$, while the other is not. This difference does not impact the proof and we can fix $Y$ and $Z$ in both ways. 

Consider the following diagrams induced by the blowup of $Z$:
\begin{equation}\label{eq:pic}
    \begin{tikzcd}
    E\ar[r,"\epsilon"]\ar[d,"\rho"] &\mathfrak{X} \ar[r,"\pi"]\ar[d,"\sigma"] & Y \\
    Z\ar[r,"\iota"] & X \ar[dotted, ru] &
    \end{tikzcd}
\end{equation}
where $\mathfrak{X}\cong \textrm{Bl}_Z X$ and the exceptional divisor $E$ is isomorphic to the partial flag variety $\mathsf{G}/(\mathsf{P}_Y\cap \mathsf{P}_Z)$. The horizontal arrow $\pi:\mathfrak{X}\rightarrow Y$ is a $\mathbb{P}^c$-bundle, which restricts to an $\mathbb{A}^c$-bundle $\pi:\mathfrak{X}\setminus E \rightarrow Y$, where $c=\dimsf  X - \dimsf  Y$.
Notice that the diagram \eqref{eq:pic} is $\mathsf{G}$-equivariant with respect to the natural actions. 

To prove \cref{prop:invariants}, we need to recall some facts. 
\begin{defn}[{\cite[\textsection~1.3.5]{eisenbud20163264}}]\label{def:strat}
A smooth variety $X$ admits a \emph{finite affine stratification} if the following holds:
\begin{enumerate}
    \item\label{cd-decomp} $X=\bigsqcup\limits_{i} U_i$, where $\{U_i\}_i$ is a finite collection of irreducible locally closed subschemes,
    \item\label{cd-closure} $\overline{U}_i=\bigsqcup\limits_{U_j\subseteq \overline{U}_i} U_j$, where $\overline{U}_i$ denotes the closure of $U_i$,
    \item\label{cd-affine} $U_i\cong \mathbb{A}^{n_i}$.
\end{enumerate}
We refer to $U_i$ as the \textit{open strata}.    
\end{defn}
Similarly to the analogous statement about Chow rings (\cf \cite[Theorem~1.1.8]{eisenbud20163264}), we recall the following well known fact, which is a consequence of Quillen's localisation sequence.\begin{theorem}\label{thm:Quillen}
Let $X$ be a smooth variety admitting a finite affine stratification $\{U_i\}_i$. Let ${\{F_i\}_i}$ be a finite collection of objects in $\DX$ such that:
\begin{itemize}
    \item $\supp F_i \subseteq \overline{U}_i$,
    \item $\ranksf \, (F_i \vert_{U_i}\,)=1$.
\end{itemize}
Then, $K_0(X)$ is a free abelian group and $\{[F_i]\}_i$ is a basis of $K_0(X)$.
\end{theorem}
We are ready for the main result of this appendix. 
\begin{proposition}\label{prop:invariants}
Let $X$ be a $\mathsf{G}$-horospherical variety of Picard rank one, with closed $\mathsf{G}$-invariant subvarieties $Y, \, Z\subset X$, then the canonical map induced by \eqref{eq:pic}:
\begin{equation}\label{eq:K_0-map}
K_0(Z)\oplus K_0(Y) \xrightarrow {(\imath_{*},\sigma_{*}\pi^{*})}  K_0(X).
\end{equation}
is an isomorphism. In particular, $K_0(X)$ is a free abelian group and we have:
\begin{equation*}
    \ranksf {K_0(X)} = \ranksf K_0(Y) + \ranksf K_0(Z).
\end{equation*}
\end{proposition}

\begin{proof}
We prove that the morphism \eqref{eq:K_0-map} is an isomorphism.
Notice that the morphisms in \eqref{eq:pic} are equivariant with respect to the action of $\mathsf{G}$.
Since $E, Y$ and $Z$ are $\mathsf{G}$-homogeneous varieties, fixing a Borel subgroup $\mathsf{B}\subset \mathsf{G}$, we obtain compatible finite affine stratifications by Schubert cells on $E$, $Y$ and $Z$. Denote their open strata respectively as $\{E_i\}_i$, $\{Y_j\}_j$ and $\{Z_k\}_k$. Notably, the Schubert cells are exactly the orbits of $\mathsf{B}$. As $E\cong \mathsf{G}/(\mathsf{P}_Y\cap \mathsf{P}_Z)$, the Schubert cells are compatible with the morphisms in \eqref{eq:pic}, i.e. for every $i$, there exist $j, k$ such that:
\begin{equation}\label{eq:compatib}
    (\pi\circ \epsilon) \, (E_i)=Y_{j} \quad \textrm{and} \quad \rho \, (E_i) =Z_k.
\end{equation}  
We now define an affine stratification of $\mathfrak{X}$. Consider the following collection of locally closed subsets:
\begin{equation*}
    U^1_j= \{ \pi^{-1}(Y_j) \setminus E \}_j \quad \textrm{and} \quad U^2_i = \{ \epsilon \, (E_i) \}_i.
\end{equation*}
We claim that the collection defined by $\{U^1_j\}_j \sqcup \{U^2_i\}_i$ is a finite affine stratification of $\mathfrak{X}$. Indeed, it is immediate to see that the proposed collection is a partition of $\mathfrak{X}$, so that \cref{def:strat}.\eqref{cd-decomp} holds. The verification of the conditions \cref{def:strat}.\eqref{cd-closure} and \cref{def:strat}.\eqref{cd-affine} is straightforward for $\{U_i^2\}_i$, hence we point out the less intuitive steps for the subcollection $\{U^1_j\}_j$. Condition \cref{def:strat}.\eqref{cd-affine} holds because $U^1_j\cong \mathbb{A}^{n_j}$, as $\pi:\mathfrak{X}\setminus E\rightarrow Y$ is an $\mathbb{A}^c$-bundle, which is trivial on every Schubert cell of $Y$.
We focus on \cref{def:strat}.\eqref{cd-closure}. From \eqref{eq:compatib}, we have:
\begin{equation*}
    (\pi\circ \epsilon)^{-1}(Y_j) = \bigsqcup_{(\pi \circ \epsilon)(E_s) \subseteq Y_j} E_s
\end{equation*}
As a consequence, we obtain:
\begin{equation}\label{eq:closure}
\begin{aligned}
    \overline{U}^1_j &= \pi^{-1}(\overline{Y}_j) =\\ &=  \bigsqcup_{Y_t \subseteq \overline{Y}_j}\bigg((\pi^{-1}(Y_t) \setminus E) \sqcup  (\pi\circ \epsilon)^{-1}(Y_t) \bigg) = \\ &=  \bigsqcup_{Y_t \subseteq \overline{Y}_j}\bigg((\pi^{-1}(Y_t) \setminus E) \sqcup \bigg(\bigsqcup_{(\pi \circ \epsilon)(E_s) \subseteq Y_t} \epsilon (E_s)\bigg) \bigg)
\end{aligned}
\end{equation}
proving \cref{def:strat}.\eqref{cd-closure}.

We now show that the collection given by 
$\{\sigma(U^1_j)\}_j\sqcup \{\imath(Z_k)\}_k$ is a finite affine stratification of $X$.  
Again, the verification of the conditions in \cref{def:strat}(\ref{cd-closure},~\ref{cd-affine}) is straightforward for $\{\imath(Z_k)\}_k$. We point out the less intuitive steps for the strata $\{\sigma(U^1_j)\}_j$. Notice that $\sigma$ is an isomorphism outside of the exceptional locus, hence \cref{def:strat}.\eqref{cd-affine} holds by definition of $U^1_j$. To show \cref{def:strat}.\eqref{cd-closure}, we have:
\begin{align*}
\overline{\sigma(U^1_j)} &= \sigma(\overline{U}^1_j)=\\
&=  \bigsqcup_{Y_t \subseteq \overline{Y}_j}\bigg(\sigma(\pi^{-1}(Y_t) \setminus E) \sqcup \bigg(\bigsqcup_{(\pi \circ \epsilon)(E_s) \subseteq Y_t} (\sigma \circ \epsilon) (E_s)\bigg) \bigg),
\end{align*}
where the first equality holds by properness of $\sigma$ and the second is a consequence of \eqref{eq:closure}. As $\sigma \circ \epsilon = \imath \circ \rho$, we conclude by \eqref{eq:compatib} that $\{\sigma(U^1_j)\}_j\sqcup\{\imath(Z_k)\}_k$ is a finite affine stratification of $X$.

The collection $\{\sigma_*\pi^* \cO_{\overline{Y}_j}\}_j\sqcup \{\imath_* \cO_{\overline{Z}_k}\}_k$ induces a basis of $K_0(X)$ by \cref{thm:Quillen}. Indeed, the support property is immediate, moreover, we have
\begin{equation*}
    (\sigma_*\pi^* \cO_{\overline{Y}_j})\vert_{\sigma(U_j^1)} = \cO_{\sigma(U_j^1)}, \quad \quad \quad (\imath_* \cO_{\overline{Z}_k})\vert_{Z_k} = \cO_{Z_k}.
\end{equation*}  

Finally, this shows that the morphism  defined in  \eqref{eq:K_0-map} carries a basis of $K_0(Z)\oplus K_0(Y)$ to a basis of $K_0(X)$, proving that $(\imath_*,\sigma_*\pi^*)$ is an isomorphism.
\end{proof}

\newpage

\bibliography{biblio}

\Addresses

\end{document}